\documentclass[10pt]{amsart}

\usepackage{amsmath}
\usepackage{amscd}
\usepackage{amsthm}
\usepackage{amsfonts}
\usepackage{amssymb}
\usepackage{amsthm}
\usepackage{subfigure}
\usepackage{graphicx}

\newcommand{\RR}{\mathbb{R}}

\newcommand{\AAA}{\mathcal{A}}
\newcommand{\BBB}{\mathcal{B}}
\newcommand{\CCC}{\mathcal{C}}
\newcommand{\III}{\mathcal{I}}
\newcommand{\RRR}{\mathcal{R}}
\newcommand{\XXX}{\mathcal{X}}

\newcommand{\pp}{\mathbf{p}}
\newcommand{\qq}{\mathbf{q}}
\newcommand{\rr}{\mathbf{r}}
\newcommand{\uu}{\mathbf{u}}

\newcommand{\irr}[1]{\mathrm{PAIR}(#1)}

\newcommand{\perm}[1]{\mathrm{SYM}(#1)}
\newcommand{\Eperm}[1]{\mathrm{ALT}(#1)}

\newcommand{\Nof}[1]{\Pi_{#1}}
\newcommand{\Spinof}[1]{\mathrm{Spin}(#1)}

\newcommand{\Sof}[1]{\sigma_{#1}}
\newcommand{\Xof}[1]{m_{#1}}
\newcommand{\Yof}[1]{\tilde{\sigma}_{#1}}
\newcommand{\Pof}[1]{\mathcal{P}_{#1}}
\newcommand{\Zof}[1]{\mathcal{Z}_{#1}}

\newcommand{\Namings}[1]{\Gamma_{#1}}

\newcommand{\ClassLab}[1]{\RRR^{\mathrm{lab}}(#1)}
\newcommand{\ClassNon}[1]{\RRR(#1)}
\newcommand{\ExtClassLab}[1]{\RRR^{\mathrm{lab}}_{\mathrm{ex}}(#1)}
\newcommand{\ExtClassNon}[1]{\RRR_{\mathrm{ex}}(#1)}

\newcommand{\simperm}{\sim}
\newcommand{\eps}{\varepsilon}

\newcommand{\RHScase}[1]{\left\{\begin{array}{ll} #1\end{array}\right.}
\newcommand{\mtrx}[1]{\left(\begin{matrix} #1 \end{matrix}\right)}
\newcommand{\bmtrx}[1]{\left[\begin{matrix} #1 \end{matrix}\right]}
\newcommand{\LL}[2]{\begin{matrix}{#1}\\{#2}\end{matrix}}

\newcommand{\Ext}[4]{\mathrm{INS}_{#1,#2,#3}(#4)}

\newtheorem{thm}{Theorem}[section]
\newtheorem{prop}[thm]{Proposition}
\newtheorem{lem}[thm]{Lemma}
\newtheorem{defn}[thm]{Definition}
\newtheorem{cor}[thm]{Corollary}
\newtheorem{rmk}[thm]{Remark}
\newtheorem{exam}[thm]{Example}

\numberwithin{equation}{section}

\author[JF]{Jon Fickenscher}
\title[Extended Classes]{Labeled and Non-labeled Extended Rauzy Classes}

\begin{document}

	\begin{abstract}
		Multiple conventions have been adopted for denoting Interval Exchange Transformations (IETs). The ``non-labeled'' convention was the original, while the ``labeled'' convention has proven convenient when investigating Flat Surfaces as described by IETs. We establish the relationship between Extended Rauzy Classes, an equivalence class of combinatorial data related to IETs, under each convention.
	\end{abstract}

	\maketitle
	
	\section{Introduction}
	
	We begin as most papers on the subject tend to, an Interval Exchange Transformation $T$ is defined by objects: a permutation $\pi$ on $N$ symbols and a vector of lengths $\lambda$. However, how $\pi$ is actually defined varies and typically conforms to one of two conventions.
	In the ``non-labeled'' convention, $\pi$ is a permutation on $\{1,\dots,N\}$ that describes the reordering of sub-intervals $I_1,\dots, I_N$ of $I$, where $I_i$ is of length $\lambda_i$ for $\lambda = (\lambda_1,\dots,\lambda_N)\in\RR_+^N$.
	In the ``labeled'' convention, we instead have an alphabet $\AAA$ of $N$ symbols. The lengths instead are encoded by $\lambda \in \RR_+^\AAA$ and $\pi=(\pi_0,\pi_1)$ is actually a pair of bijections $\pi_\eps:\AAA\to\{1,\dots,N\}$ that denote the order of sub-intervals $\{I_a\}_{a\in\AAA}$ before and after $T$ acts on $I$, where each $I_a$ is length $\lambda_a$.
	
	Under either convention, Rauzy Induction (on the right) defines equivalence classes for $\pi$'s called \emph{Rauzy Classes} (see \cite{cRau1979}). An \emph{Extended Rauzy Class} is an equivalence class on $\pi$ by Rauzy Induction on both the left and right.
	
	In \cite{cBoi2013}, C. Boissy determined the degree of covering from a ``labeled'' Rauzy Class to its corresponding ``non-labeled'' Rauzy Class.
	
	In this work, we provide the degree of covering from a ``labeled'' Extended Rauzy Class to its corresponding ``non-labeled'' Extended Rauzy Class. In particular, we show the following.
	
		\begin{thm}\label{ThmMain}
			Let $\pi$ be a ``labeled'' irreducible permutation on $\AAA$, $N=\#\AAA$. If any Flat Surface that arises from $\pi$ has singularities with degrees $\ell_1,\dots,\ell_n$, then
				$$ \frac{\#\ExtClassLab{\pi}}{\#\ExtClassNon{\pi}} = \RHScase{N!, & \mbox{if }\ell_i=\ell_j\mbox{ for some }i\neq j\\ N!/2, & \mbox{all }\ell_j\mbox{ are distinct}.}$$
		\end{thm}
	
	\section{Definitions}\label{SEC_DEF}
	
		In this section, we will provide notation, conventions and definitions used in this paper. Once we establish our use of the symmetric and alternating groups on finite alphabets. We will then define ``labeled'' permutations, which we will call \emph{pairs}, and define Rauzy Induction and (Extended) Rauzy Classes in this context. We then define the set of \emph{renamings} of a pair, the main focus of our study.
		
		From \cite{cFick2012}, we recall \emph{switch moves}, a sort of acceleration of Rauzy Induction. We also state some key features of switches as provided in that paper. Given a pair (``labeled'' permutation) $\pp$, we define objects $\Sof{\pp}$, $\Nof{\pp}$ and $\Spinof{\pp}$. These objects relay information concerning the flat surfaces associated to $\pp$. We will mention this relationship for the reader with background in this area.
		
		We end the section by restating Theorem \ref{ThmMain} in terms of these objects.

		\subsection{The symmetric and alternating groups}
	
			Given a group $G$, we denote by $<b,c>$ the subgroup of $G$ generated by $b,c\in G$. We extend this definition to allow $<b,H>$ to be the subgroup generated by $b\in G$ and the elements in $H\subseteq G$. We will therefore allow generating sets to be combinations of elements and subgroups. 
			Given finite alphabet $\AAA$, let $\perm{\AAA}$ denote the \emph{symmetric group} of permutations on $\AAA$. Let $\Eperm{\AAA}$ denote the \emph{alternating subgroup} of $\perm{\AAA}$.
			For $\mu,\nu\in\perm{\AAA}$, we use composition as group multiplication,
			      \begin{equation}
				    \mu\cdot \nu = \mu\circ\nu.
			      \end{equation}
			We often represent elements of $\perm{\AAA}$ in \emph{cycle notation}.
		For example,
			if $\AAA = \{a,b,c,d,e\}$,
			      $
				    \mu = (a,b,c,d)\mbox{ and } \nu = (a,d,e)(b,c),
			      $
			then
			      $
				     \mu\nu = (b,d,e)\mbox{ and } \nu\mu = (a,c,e).
			      $
			
		\begin{rmk}
			By using the natural embedding, we will always assume $\perm{\BBB}\subseteq \perm{\AAA}$ for subalphabet $\BBB\subseteq \AAA$.
		\end{rmk}

		\subsection{Pairs, Rauzy Induction and Classes}

		\begin{defn}
		      Given finite alphabet $\AAA$ of size $N$, $\irr{\AAA}$ is the collection of pairs $\pp = (p_0,p_1)$ of bijections $p_\eps:\AAA \to \{1,2,\dots,N\}$, $\eps\in\{0,1\}$, that satisfy
		      	\begin{equation}\label{EqIrr}
		      		p_1\circ p_0^{-1}(\{1,2,\dots,k\}) = \{1,2,\dots,k\}
		      	\end{equation}
		      if and only if $k=N$.
		\end{defn}
		This is typically considered to be collection of \textit{irreducible permutations} when considering permuations on alphabet $\AAA$. See \cite{cBuf2006}, \cite{cMarMouYoc2005} and \cite{cVia2006} for uses of this convention. However, we wish to make the distinction between such ``permutations'' and bijections on $\{1,\dots,N\}$.

		If $\nu:\AAA'\to\AAA$ is a bijection and $\pp=(p_0,p_1)\in\irr{\AAA}$, then we define
		  \begin{equation}
		      \pp\circ\nu = (p_0\circ \nu,p_1\circ \nu) \in \irr{\AAA'}.
		  \end{equation}

		Because this paper is focused on combinatorics, we will simply define Rauzy Induction from the left and right as bijections on $\irr{\AAA}$. If $\pp\in\irr{\AAA}$, 
		      \begin{equation*}p_0^{-1}(1) \neq p_1^{-1}(1)\mbox{ and }p_0^{-1}(N)\neq p_1^{-1}(N),\end{equation*}
		where  $N=\#\AAA$, by Equation \eqref{EqIrr}. Therefore the following are well defined. Also, the fact that Ruazy Induction on $\pp\in\irr{\AAA}$ yields another element of $\irr{\AAA}$ is well known and will be accepted in the definition. We refer the reader to \cite{cRau1979} and \cite{cVia2006} for  definitions of Rauzy Induction motivated in terms of IETs.

		\begin{defn}
		      Let $\pp=(p_0,p_1)\in\irr{\AAA}$.

		      The result of \emph{type $\eps$ (right) Rauzy Induction} on $\pp$, $\eps\in\{0,1\}$, is $\pp'=(p_0',p_1')\in \irr{\AAA}$ given by
			    \begin{equation}
				  p'_\eps = p_\eps\mbox{ and }
				      p_{1-\eps}'(b) = \RHScase{p_{1-\eps}(b), & p_{1-\eps}(b) \leq p_{-1\eps}(z_\eps),\\
						p_{1-\eps}(z_\eps)+1, & b=z_{1-\eps},\\
						p_{1-\eps}(b)+1, & p_{1-\eps}(z_\eps)<p_{1-\eps}(b)<N,}
			    \end{equation}
		      where $z_\eps = p_\eps^{-1}(N)$ and $N=\#\AAA$.

		      The result of \emph{type $\eps$ left Rauzy Induction} on $\pp$, $\eps\in\{0,1\}$, is $\pp''=(p_0'',p_1'')\in \irr{\AAA}$ given by
			    \begin{equation}
				  p''_\eps = p_\eps\mbox{ and }
				      p_{1-\eps}''(b) = \RHScase{p_{1-\eps}(b), & p_{1-\eps}(b) \geq p_{-1\eps}(a_\eps),\\
						p_{1-\eps}(z_\eps)-1, & b=a_{1-\eps},\\
						p_{1-\eps}(b)-1, & 1<p_{1-\eps}(b)<p_{1\eps}(a_\eps),}
			    \end{equation}
		      where $a_\eps = p_\eps^{-1}(1)$.
		\end{defn}

		From the definition of induction, for any $\pp\in \irr{\AAA}$ and $\pp' = \omega \pp$ for choice of induction $\omega$, there exists $n>0$ so that $\omega^n \pp' = \pp$. It follows that there is a choice of inductive moves from $\pp$ to $\qq$ if and only if there is a choice of moves from $\qq$ to $\pp$. Therefore the following definition does indeed describe \emph{equivalence classes} under induction.

		\begin{defn}
		      Let $\pp\in\irr{\AAA}$.
			
		      The \emph{(``labeled'') Rauzy Class} of $\pp$, $\ClassLab{\pp}\subset\irr{\AAA}$, is the orbit of both types of right Rauzy Induction on $\pp$. The \emph{(``labeled'') Extended Rauzy Class} of $\pp$, $\ExtClassLab{\pp}\subset\irr{\AAA}$, is the orbit of both types of left and right Rauzy Induction.

		      The \emph{``non-labeled'' Rauzy Class} (resp. \emph{Extended Rauzy Class}) of $\pp$, $\ClassNon{\pp}$ (resp. $\ExtClassNon{\pp}$), is the image of $\ClassLab{\pp}$ (resp. $\ExtClassLab{\pp}$) under the map $\pp \mapsto p_1\circ \pi_0^{-1}$. These are subsets of $\perm{\{1,\dots,N\}}$ where $N=\#\AAA$.
		\end{defn}

		Because these sets are all equivalence classes, they do not depend on initial choice of $\pp$. For example, $\ClassLab{\qq} = \ClassLab{\pp}$ for any $\qq\in \ClassLab{\pp}$.

		\begin{defn}
		  $\pp\in\irr{\AAA}$ is \emph{standard} if and only if
		  	\begin{equation}
		  		p_0\circ p_1^{-1}(1) = p_1\circ p_0^{-1}(1) = N
		  	\end{equation}
		  	where $N=\#\AAA$.
		\end{defn}

		The following result was shown first in \cite[Section 35]{cRau1979}.
		
		\begin{prop}
		      Every Rauzy Class contains a standard pair.
		\end{prop}

		\subsection{Remanings}

		In order to prove Theorem \ref{ThmMain}, we will define an object that describes the covering of $\ExtClassNon{\pp}$ by $\ExtClassLab{\pp}$. We first motivate the definition: for $\pp,\qq\in\irr{\AAA}$
		    \begin{equation}
			  q_1\circ q_0^{-1} = p_1\circ p_0^{-1} \iff \qq = \pp \circ \nu \mbox{ for some }\nu\in\perm{\AAA}.
		    \end{equation}
		In other words, $\pp$ and $\qq$ are in the same fiber over $\ExtClassNon{\pp}$ if and only if $\pp$ and $\qq$ are the same \emph{up to a renaming}. We want to describe all such $\nu$ that may occur for $\pp$ such that $\pp\circ \nu\in\ExtClassLab{\pp}$.

		\begin{defn}
		    If $\pp=(p_0,p_1)\in\irr{\AAA}$, then
			  $$ \Namings{\pp} := \left\{\nu\in\perm{\AAA}: \qq = \pp\circ\nu,~\qq\in\ExtClassLab{\pp}\right\}$$
		    is the set of \emph{renamings} of $\pp$ in $\ExtClassLab{\pp}$.
		\end{defn}

		The following lemma establishes that renamings actually form a group. Furthermore, this group does not depend on which representative $\pp\in\irr{\AAA}$ is chosen.

		\begin{lem}\label{LemNamingsClass}
		    Let $\pp\in\irr{\AAA}$. $\Namings{\pp}$ is a subgroup of $\perm{\AAA}$, and if $\qq\in\ExtClassLab{\pp}$, then $\Namings{\pp} = \Namings{\qq}$.
		\end{lem}

		\begin{proof}
			We first show that $\Namings{\pp}=\Namings{\qq}$. By definition, it suffices to show this statement if $\qq$ is the result of exactly one move Rauzy Induction.
			
			By verification, we see that if $\rr = \pp\circ \nu\in\ExtClassLab{\pp}$ and $\uu$ is the result of the same move of Ruazy Induction on $\rr$, then $\uu = \qq\circ \nu$. Because $\uu = \qq\circ \nu \in \ExtClassLab{\pp} = \ExtClassLab{\qq}$, $\nu\in\Namings{\qq}$. We then conclude that $\Namings{\pp}\subseteq \Namings{\qq}$. Equality follows as Extended Rauzy Classes are two-way connected.
			
			We now may verify that $\Namings{\pp}$ is a subgroup of $\perm{\AAA}$. Because $\pp \in \ExtClassLab{\pp}$, the identity element belongs to $\Namings{\pp}$. Suppose $\mu,\nu\in\Namings{\pp}$. Then $\qq=\pp\circ \mu\in \ExtClassLab{\pp}$. Because $\nu\in\Namings{\pp}=\Namings{\qq}$, $\qq\circ \nu = \pp\circ(\mu\nu)\in\ExtClassLab{\qq}=\ExtClassLab{\pp}$ or $\mu\nu\in\Namings{\pp}$.
%		      Because $\pp\in \ExtClassLab{\pp}$, $\Namings{\pp}$ contains the identity element. For every $\nu\in\Namings{\pp}$, there is a \emph{Rauzy Path} $\omega_\nu$, or a sequence of choices for type and side (left/right) Ruazy Inudction moves, that lead from $\pp$ to $\pp\circ\nu$. Similarly, there is a path $\omega_\mu$ for $\mu\in\Namings{\pp}$ from $\pp$ to $\pp\circ\mu$. We may verify that the choices $\omega_\nu$ may be applied to $\pp\circ \mu$ and the resulting pair is $\pp\circ\nu\circ\mu$. Therefore $\nu\mu\in\Namings{\pp}$, and $\Namings{\pp}$ is a subgroup of $\perm{\AAA}$.
		\end{proof}

		\subsection{Switch Moves}
		
		Switch moves were defined in \cite{cFick2012}. We will mention relevant results concerning these moves here, as they will be recalled in later sections. The first type of switch move travels from one standard $\pp$ to another in $\ClassLab{\pp}$.

		\begin{defn}
		      Let $\pp=(p_0,p_1)\in\irr{\AAA}$ be standard. If there exists $b,c\in\AAA$ so that
			    \begin{equation}
				  p_0(b) < p_0(c) \mbox{ and } p_1(b)<p_1(c),
			    \end{equation}
		      we express $\pp$ as
			    \begin{equation*}
			      \pp = \mtrx{\LL{a}{z}\; \LL{\overline{u_0}}{\overline{u_1}}\;\LL{b}{b}\;\LL{\overline{v_0}}{v_1}\;\LL{c}{c}\;\LL{\overline{w_0}}{\overline{w_1}}\;\LL{z}{a}}.
			    \end{equation*}
		      A \emph{$\{b,c\}$-switch move} on $\pp$ is a choice of Rauzy Induction moves on $\pp$ resulting in $\pp'\in \ClassLab{\pp}$, where
			    \begin{equation*}
			      \pp' = \mtrx{\LL{a}{z}\; \LL{\overline{v_0}}{v_1}\;\LL{c}{c}\;\LL{\overline{u_0}}{\overline{u_1}}\;\LL{b}{b}\;\LL{\overline{w_0}}{\overline{w_1}}\;\LL{z}{a}}.
			    \end{equation*}
		      We call such a switch an \emph{inner switch}.
		\end{defn}

		The second type of switch move travels from standard $\pp$ to another in $\ExtClassLab{\pp}$. Note that the resulting pair \emph{cannot} be an element of $\ClassLab{\pp}$ as one of the letters that begins a row differs from the corresponding row in $\pp$.

		\begin{defn}
		    Let $\pp=(p_0,p_1)\in\irr{\AAA}$ be standard, $a=p_0^{-1}(1)$, $z = p_1^{-1}(1)$, and $d\in\AAA\setminus\{a,z\}$. Express $\pp$ as
			    \begin{equation}
				  \pp = \mtrx{\LL{a}{z}\;\LL{\overline{u_0}}{\overline{u_1}}\;\LL{d}{d}\;\LL{\overline{v_0}}{\overline{v_1}}\;\LL{z}{a}}.
			    \end{equation}

		    An \emph{$\{a,d\}$-switch move} on $\pp$ is any choice of left and/or right Rauzy Induction moves on $\pp$ resulting in $\pp'' \in \ExtClassLab{\pp}$, where
			  \begin{equation*}
				\pp'' = \mtrx{\LL{d}{z}\;\LL{\overline{v_0}}{\overline{v_1}}\;\LL{a}{a}\;\LL{\overline{u_0}}{\overline{u_1}}\;\LL{z}{d}}.
			  \end{equation*}

		    A \emph{$\{d,z\}$-switch move} on $\pp$ is any choice of left and/or right Rauzy Induction moves on $\pp$ resulting in $\pp''' \in \ExtClassLab{\pp}$, where
			  \begin{equation*}
				\pp''' = \mtrx{\LL{a}{d}\;\LL{\overline{v_0}}{\overline{v_1}}\;\LL{z}{z}\;\LL{\overline{u_0}}{\overline{u_1}}\;\LL{d}{a}}.
			  \end{equation*}
		    We call either switch an \emph{outer switch}.
		\end{defn}

		The following result from \cite{cFick2012} tells us that in order to find all standard pairs in a Rauzy Class or Extended Rauzy Class, it suffices to consider all standard pairs connected by switch moves.
		
		\begin{prop}
			Let $\pp,\qq\in\irr{\AAA}$ be distinct standard pairs. Then $\qq\in \ClassLab{\pp}$ if and only if $\pp$ and $\qq$ are related by an inner switch path, one composed of only inner switches. Also, $\qq\in\ExtClassLab{\pp}$ if and only if they are related by a switch path, which may be composed of both inner and outer switches.
		\end{prop}

		\subsection{Surface Properties}
 
		The following are combinatorial objects that describe various properties related to Translation Surfaces that arise from $\pp$. The remarks included are meant to relate the information here directly to those properties. The reader may choose to skip these remarks if desired, as the combinatorial information is the only necessary component for this work.

		\begin{defn}
			If $\pp=(p_0,p_1)\in\irr{\AAA}$, $N=\#\AAA$, then $\Sof{\pp}\in \perm{\AAA}$ is given by
				\begin{equation}\label{eq_SofPi}
					\Sof{\pp}(a) = \RHScase{p_0^{-1}(1), & p_1(a)=1,\\
						p_0^{-1}\Big(p_0\big(p_1^{-1}(N)\big)+1\Big), & p_1(a) = p_1(p_0^{-1}(N))+1,\\
						p_0^{-1}\Big(p_0\big(p_1^{-1}\big(p_1(a)-1\big)\big)+1\Big), & \mathrm{otherwise.}}
				\end{equation}
		\end{defn}

		\begin{rmk}\label{RMK_SofP}
		      This is similar in construction to $\sigma$ as defined in \cite[Section 2]{cVe1982}. The object here is designed to describe the following (given a surface associated to $\pp$): for every $b\in \AAA \setminus \{a\}$, $p_0(a) = 1$, travel clock-wise around the left vertex $v$ of the edge labeled $b$ until first arriving at a vertical saddle connection of $v$. Name a small segment that begins at $v$ of this sadde connection $b$. For the letter $a$, instead travel until reaching the first horizontal saddle connection, and name a small segment of this connection beginning at $v$ with the letter $a$. See Figure \ref{FIG_S_OF_PI}a for an example.

		      Then $\Sof{\pp}$ describes the order these named segments are encountered when moving in a clock-wise direction around each vertex $v$ (see Figure \ref{FIG_S_OF_PI}b for an example). Note that if there are $\ell+1$ different \emph{vertical} segments beginning at vertex $v$, $v$ represents a \emph{singularity of degree} $\ell$. Therefore, the cycles of $\Sof{\pp}$ are in bijective correspondence to the singularities of the surface.

		      While the letter $a$ is not needed in the above description, it does ``mark'' a singularity. This choice of singularity remains fixed throughout pairs in a Rauzy Class (see \cite{cBoi2012}).
		\end{rmk}

		\begin{figure}[t!]
		    \begin{center}
		    \subfigure[A surface related to $\pp$ with named segments.]{\includegraphics[width = .51\textwidth]{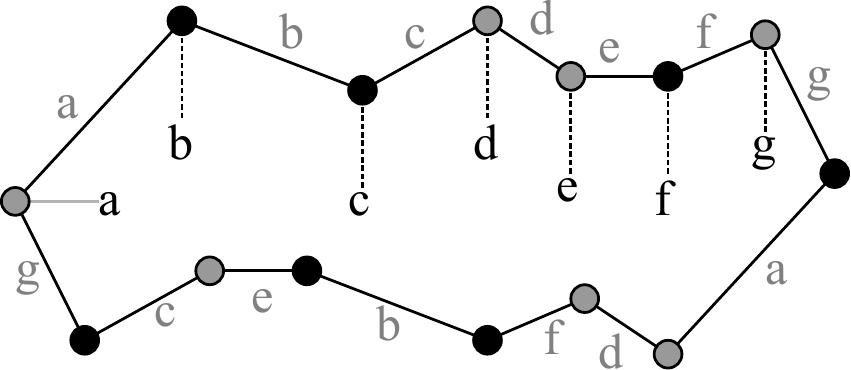}}
		    \hfill
		    \subfigure[Picture of $\Sof{\pp}$.]{\includegraphics[width = .279\textwidth]{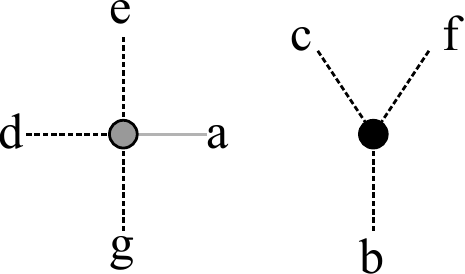}}
		    \end{center}
		    \caption{A surface associated to $\pp$ from Example \ref{ex_SofPi} and the permtuation $\Sof{\pp}$.}\label{FIG_S_OF_PI}
		\end{figure}

%		\begin{rmk}
%			If $\pp=(p_0,p_1)\in\irr{\AAA}$ is standard, then $\Sof{\pp}$ is given by
%				\begin{equation}\label{eq_SofPi2}
%					\Sof{\pp}(a) = \RHScase{p_0^{-1}(1), & p_1(a)=1,\\
%						p_0^{-1}(2), & p_1(a) = 2,\\
%						p_0^{-1}\Big(p_0\big(p_1^{-1}\big(p_1(a)-1\big)\big)+1\Big), & \mathrm{otherwise.}}
%				\end{equation}			
%		\end{rmk}
		
		\begin{exam}\label{ex_SofPi}
			If
				$$ \pp = \mtrx{\LL{a}{g}\;\LL{b}{c}\;\LL{c}{e}\;\LL{d}{b}\;\LL{e}{f}\;\LL{f}{d}\;\LL{g}{a}},$$
			then $\Sof{\pp} = (a,e,d,g)(b,f,c)$.
		\end{exam}

		As indicated in Remark \ref{RMK_SofP}, we will want to extract two properties from $\Sof{\pp}$: the letter $a\in \AAA$ so that $p_0(a)=1$ and the action of $\Sof{\pp}$ restricted to $\AAA\setminus \{a\}$. The object $\Nof{\pp}$ defined here achieves these objectives.
		
		\begin{defn}
			If $\pp=(p_0,p_1)\in\irr{\AAA}$, then $\Nof{\pp}=(\Xof{\pp},\Yof{\pp})\in \AAA \times \perm{\AAA}$ is given by $\Xof{\pp} = p_0^{-1}(1)$ and $\Yof{\pp}$ is given by
				\begin{equation}
					\Yof{\pp}(a) = \RHScase{\Xof{\pp}, & a = \Xof{\pp},\\ 
					      \Sof{\pp}\circ\Sof{\pp}(a), & \Sof{\pp}(a) = \Xof{\pp},\\
					      \Sof{\pp}(a), & \mathrm{otherwise.}}
				\end{equation}
			We will denote $\Nof{\pp}$ by
				$$ \Nof{\pp} = [\Xof{\pp}] \Yof{\pp}$$
			where $\Yof{\pp}$ is given in cycle notation and the trivial cycle $(\Xof{\pp})$ is ommitted as $[\Xof{\pp}]$ is already included.
		\end{defn}
		
		\begin{exam}
			If $\pp$ is from Example \ref{ex_SofPi}, then $\Nof{\pp} = [a](b,f,c)(e,d,g)$.
		\end{exam}

		\begin{rmk}\label{RMK_CYCLES}
		    By Remark \ref{RMK_SofP}, we conclude that the cycles of $\Nof{\pp}$ (except $[\Xof{\pp}]$) now relate uniquely to a singularity \emph{and} the cycle has length $\ell+1$, where $\ell$ is the degree of the singularity.
		\end{rmk}

		There statement of Theorem \ref{ThmMain} uses the numbers of degrees of singualrities to determine which value $N!/2$ or $N!$ is achieved. By Remarks \ref{RMK_SofP} and \ref{RMK_CYCLES}, we may instead speak of the length of cycles in $\Nof{\pp}$.

		\begin{defn}
		      If $\pp\in\irr{\AAA}$, then $\Pof{\pp}$ is the unordered list of cycles lengths with multiplicity in $\Yof{\pp}$, excluding $(\Xof{\pp})$. We call $\Pof{\pp}$ \emph{simple} if each value is multiplicity one.
		\end{defn}
		
		\begin{exam}
			If $\pp$ is from Example \ref{ex_SofPi}, then $\Pof{\pp} = \{3,3\}=\{3^2\}$ and is not simple.
		\end{exam}

		\begin{rmk}\label{RMK_PofP}
		      We now see directly that $\Pof{\pp} = \{\ell_1+1,\dots, \ell_n+1\}$ where $\ell_1,\dots,\ell_n$ are the degrees of the singularities with multiplicity that occur in any Translation Surface associated to $\pp$.
		\end{rmk}

		In this work, we will use the \emph{parity of spin structure}, $\Spinof{\pp}$, for a pair $\pp$. This is a value in $\{0,1\}$ mentioned in \cite{cKonZor2003} that distinguished Extended Rauzy Classes. We will not discuss the formal definition here but rather encourage the reader to look at \cite{cKonZor2003} and \cite{cZor2008} as references for the following result.
				
		\begin{prop}\label{PropSpinDefined}
			Let $\pp\in\irr{\AAA}$. $\Spinof{\pp}$ is well defined if and only if $\Pof{\pp}$ contains only odd values.
		\end{prop}

		\subsection{Reduction of Theorem \ref{ThmMain}}

		We now may present a more direct theorem given the definitions from this section. As the proof indicates, this result follows from Propositions \ref{PropNamingsContainEven} and \ref{PropFullGroup} that follow. We then show that this result implies Theorem \ref{ThmMain}.

		\begin{thm}\label{ThmMain2}
			If $\pp\in\irr{\AAA}$, then
				$$ \Namings{\pp} = \RHScase{\Eperm{\AAA}, & \Pof{\pp}\mbox{ is simple}\\
										\perm{\AAA}, & \mbox{otherwise.}}$$
		\end{thm}
		
		\begin{proof}
			By Proposition \ref{PropNamingsContainEven}, $\Eperm{\AAA}\subseteq\Namings{\pp}$. By Proposition \ref{PropFullGroup}, if $\Pof{\pp}$ is simple $\Namings{\pp} = \Eperm{\AAA}$ and if $\Pof{\pp}$ is not simple, $\Namings{\pp}\subsetneq \Eperm{\AAA}$. Becuse $\Namings{\pp}$ is a group by Lemma \ref{LemNamingsClass}, $\Namings{\pp} = \Eperm{\AAA}$ in this case.
		\end{proof}
		
		\begin{proof}[Proof of Theorem \ref{ThmMain}]
			Consider ``labeled'' permutation $\pp$ and ``non-labeled'' $\pi = p_1\circ p_0^{-1}$. If $\pi'$ is the ``non-labeled'' Extended Rauzy Class of $\pi$, $\ExtClassNon{\pi}$, then there exists $\pp'\in\ExtClassLab{\pp}$ so that $\pi'=p'_1\circ [p'_0]^{-1}$. So for every $\pi'$ with associated to $\pp'$, there are exactly $\#\Namings{\pp'}$ distinct $\qq\in\ExtClassLab{\pp}$ so that $q_1\circ q_0^{-1} = \pi'$. And so
				$$ \# \ExtClassLab{\pp} = \#\left[\bigcup_{\pi'\in\ExtClassNon{\pi}}\left\{\qq: q_1\circ q_0^{-1} = \pi'\right\}\right]$$
				$$ = \sum_{\pi'\in\ExtClassNon{\pi}} \#\Namings{\pp'} = (\#\Namings{\pp})\cdot (\#\ExtClassNon{\pi})$$
			because $\Namings{\pp'} = \Namings{\pp}$ for each $\pp'\in\ExtClassLab{\pp}$ by Lemma \ref{LemNamingsClass}.
			We see that
				$$ \frac{\#\ExtClassLab{\pp}}{\#\ExtClassNon{\pi}} = \#\Namings{\pp}.$$
			If $\ell_1,\dots,\ell_n$ are the singularities associated to $\pp$, then
				$$ \Pof{\pp} = \{\ell_1+1,\dots,\ell_n+1\}$$
			and all $\ell_j$ are distinct if and only if $\Pof{\pp}$ is simple. Therefore
				$$ \Namings{\pp} = \RHScase{\Eperm{\AAA}, & \mbox{if } \ell_i=\ell_j\mbox{ for some }i\neq j,\\
										\perm{\AAA}, & \mbox{all }\ell_j\mbox{ are distinct,}}$$
			by Theorem \ref{ThmMain2}. We conclude the proof, as $\#\Eperm{\AAA} = N!/2$ and $\#\perm{\AAA} = N!$.
		\end{proof}

	\section{Invariant Properties}

		We provide in this section results that show the invariance of many of the objects defined in Section \ref{SEC_DEF} within Rauzy and Extended Rauzy Classes. We provide references for proofs not included here.

		\begin{prop}
		    If $\qq\in\ClassLab{\pp}$ for $\pp\in \irr{\AAA}$, then $\Sof{\pp} = \Sof{\qq}$ and $\Nof{\pp} = \Nof{\qq}$.
		\end{prop}
		
		\begin{proof}
			The proof that $\Sof{\pp}=\Sof{\qq}$ is from \cite{cFick2012}. Because $\Nof{\pp}$ is defined by $\Sof{\pp}$, this result follows as well.
		\end{proof}

		\begin{defn}
		      Let $\pp\in\irr{\AAA}$ and $\nu\in\perm{\AAA}$. The \emph{right action of $\nu$ on $\Nof{\pp}$} is given by
			  $$ \Nof{\pp}\ast\nu = (\Xof{\pp}\ast\nu,\Yof{\pp}\ast \nu)$$
		      where $\Xof{\pp}\ast \nu = \nu^{-1}(\Xof{\pp})$ and $\Yof{\pp}\ast \nu = \nu^{-1}\Yof{\pp}\nu$.
		      The \emph{centralizer} of $\Nof{\pp}$ is
			    $$ \Zof{\pp} := \{\nu\in\perm{\AAA}: \Nof{\pp}\ast \nu = \Nof{\pp}\}.$$
		\end{defn}
		
		\begin{rmk}
			The definition given for the right action of $\perm{\AAA}$ on $\Nof{\pp}$ was chosen so that if $\nu,\eta\in\perm{\AAA}$ and $\pp\in\irr{\AAA}$, then
				\begin{equation}
					\left(\Nof{\pp}*\nu\right)*\eta = \Nof{\pp}*(\nu\eta).
				\end{equation}
			Therefore $\Nof{\pp}*\nu_1*\nu_2*\cdots*\nu_k = \Nof{\pp}*\nu_1\nu_2\cdots\nu_k$ is well defined, $\nu_1,\dots,\nu_k\in\perm{\AAA}$.
		\end{rmk}

		\begin{defn}\label{DefMuOmega}
		    Let $\pp\in \irr{\AAA}$ be standard and $\qq$ the result of switch move $\omega$ on $\pp$. We define the permutation $\mu_{\omega}$ by
			$$ \mu_{\omega} = (a,z,b)$$
		    if $\omega$ is an $\{a,b\}$-switch where $a = p_0^{-1}(1)$ and $z = p_1^{-1}(1)$, and $\mu_{\omega}$ is trivial otherwise.
		    If $\omega = \omega_1\dots \omega_k$ is a switch path from standard $\pp$ to $\qq$, then we define $\mu_\omega$ by
			$$ \mu_\omega = \mu_{\omega_1}\mu_{\omega_2}\cdots\mu_{\omega_k}.$$
		\end{defn}

		\begin{lem}\label{LemQuotientOverZof}
		    If $\omega$ is a switch path leading from standard $\pp$ to $\qq$, then
			    $$ \Nof{\qq} = \Nof{\pp}\ast \mu_\omega.$$
		    Furthermore, if $\pp\simperm\qq$ and $\qq = \pp\circ \nu$ for $\nu\in\perm{\AAA}$, then
			    $$ \nu = \mu_\omega \cdot\eta$$
		    where $\eta \in \Zof{\qq}$.
		\end{lem}
		
		\begin{proof}
			This may first be verified for the case when $\qq$ is related to $\pp$ by exactly one switch move. We then proceed by induction. If $\qq$ is the result of $n+1$ switches on $\pp$, let $\omega'$ be the first $n$ switches with result $\qq'$ and $\omega_{n+1}$ be the final switch from $\qq'$ to $\qq$. Then
					$$ \Nof{\qq} = \Nof{\qq'} * \mu_{\omega_{n+1}} = \big(\Nof{\pp}*\mu_{\omega'}\big)*\mu_{\omega_{n+1}} = \Nof{\pp}*\mu_{\omega},$$
			as desired.
			
			If $\qq = \pp\circ \nu\in \ExtClassLab{\pp}$, then $\qq$ is the result of a switch path $\omega$ from $\pp$. We see then that
				\begin{equation}
					\Nof{\qq} = \Nof{\pp}* \nu = \Nof{\pp} * \mu_\omega \Rightarrow \Nof{\qq}* \mu^{-1}_{\omega}\nu = \Nof{\qq},
				\end{equation}
			and so $\mu_\omega^{-1}\cdot\nu = \eta\in \Zof{\qq}$ or $\nu = \mu_\omega\cdot\eta$.
		\end{proof}

		This corollary already follows from \cite{cKonZor2003}, but we mention it here as its proof is short given the above lemma.
		
		\begin{cor}\label{CorPofInvariant}
			If standard $\qq\in \ExtClassLab{\pp}$ for standard $\pp\in\irr{\AAA}$, then $\Pof{\qq} = \Pof{\pp}$.
		\end{cor}
		
		\begin{proof}
			$\Yof{\qq}$ is conjugate to $\Yof{\pp}$ by Lemma \ref{LemQuotientOverZof}. The number and orders of cycles of a permutation is invariant under conjugation.
		\end{proof}
		
		\begin{cor}\label{CorEvenZof}
			If $\pp\in\irr{\AAA}$ and $\Zof{\pp}\subseteq\Eperm{\AAA}$, then $\Namings{\pp}\subseteq\Eperm{\AAA}$.	
		\end{cor}
		
		\begin{proof}
			By Lemma \ref{LemQuotientOverZof}, if $\nu\in\Namings{\pp}$ then $\nu = \mu_\omega\cdot \eta$ where $\mu_\omega$, as given in Definition \ref{DefMuOmega} for any switch path $\omega$ from $\pp$ to $\pp\circ\nu$, and $\eta\in\Zof{\pp\circ \nu}$ are both even.
		\end{proof}
		
		The following proposition tells us that, when it is well defined, $\Spinof{\pp}$ is an invariant of an Extended Rauzy Class. See \cite{cKonZor1997} and \cite{cDele2013} for references and proofs.

		\begin{prop}\label{prop_spin}
			If $\pp,\qq \in \irr{\AAA}$ belong to the same Extended Rauzy Class and $\Pof{\pp}=\Pof{\qq}$ contains only odd cycles (i.e. $\Spinof{\pp}$ and $\Spinof{\qq}$ are defined), then $\Spinof{\pp} = \Spinof{\qq}$.
		\end{prop}

		\subsection{Prefix Insertions}

		We define in this section the notion of prefix insertions, maps that send a pair $\pp$ in $\irr{\AAA}$ to a pair $\hat{\pp}$ on a larger alphabet. We consider the effect of switch moves on insertions. Specifically, inner switch moves commute with insertions while outer switch moves generally do not.
	
		\begin{defn}
			If $\pp=(p_0,p_1)\in\irr{\AAA}$, $b_0,b_1\in\AAA$ and $x\notin\AAA$, then $$\hat{\pp} = (\hat{p}_0,\hat{p}_1) = \Ext{x}{b}{c}{\pp}\in\irr{\AAA\cup\{x\}}$$ is a \emph{prefix insertion} on $\pp$ given by
				$$ \hat{p}_\eps(a) = \RHScase{p_\eps(a), & p_\eps(a)<p_\eps(b_\eps)\\ p_\eps(b_\eps), & a=x \\ p_\eps(a)+1, & p_\eps(a)\geq p_\eps(b_\eps).}$$
		\end{defn}
		
		\begin{lem}\label{LemPrefixes}
			If $\pp\in\irr{\AAA}$, $x_1,\dots,x_m$ are distinct letters disjoint from $\AAA$, $b_1,\dots,b_m$ are distinct letters in $\AAA$ and $c_1,\dots,c_m$ are distinct letters in $\AAA$, then the prefix insertions
				$$ \III_j = \Ext{x_j}{b_j}{c_j}{*}$$
			may be applied in any order to $\pp$. And so
				$$ \Ext{(x_1,\dots,x_m)}{(b_1,\dots,b_m)}{(c_1,\dots,c_m)}{\pp} = \III_1\circ \dots \circ \III_m(\pp)$$
			is well defined.
		\end{lem}
		
		Note that the sets $\{b_1,\dots,b_m\}$ and $\{c_1,\dots,c_m\}$ can have common elements in the above lemma. The lemma only requires that $b_i\neq b_j$ and $c_i\neq c_j$ for $i\neq j$.
		
		\begin{lem}
			If $\pp\in\irr{\AAA}$ and $\hat{\pp} = \Ext{x}{b_0}{b_1}{\pp}$, then
				$$ \Sof{\hat{\pp}}(a) = \RHScase{b_0, & a = b_1,\\ \Sof{\pp}(b_1), & a=x,\\ x, & \Sof{\pp}(a) = b_0,\\ \Sof{\pp}(a), & \mathrm{otherwise,}}$$
			if $\Sof{\pp}(b_1)\neq b_0$, while $\Sof{\pp} = \Sof{\hat{\pp}}$ on $\AAA$ and $\Sof{\hat{\pp}}(x)=x$ if $\Sof{\pp}(b_1) = b_0$.
		\end{lem}

		If we consider a pair $\pp$ and pair $\hat{\pp}$ given by a fixed insertion rule, we want to know how the same switch move $\omega$ affects the relationship between the resulting pairs $\qq = \omega \pp$ and $\hat{\qq} = \omega\hat{\pp}$. The next lemma explicitly gives the insertion rule, which may be different than the original, such that $\hat{\qq}$ is the result of an insertion on $\qq$. In particular, we note that if $\omega$ is an inner switch move then the insertion rule is unchanged.
		
		\begin{lem}\label{LemSwitchAndInsert}
			Let $\pp \in \irr{\AAA}$ be standard, $x\notin\AAA$ and $\hat{\pp} = \Ext{x}{b_0}{b_1}{\pp}$, where $b_0,b_1\in\AAA$ with $p_\eps(b_\eps)>1$, $\eps\in\{0,1\}$. Let $a = p_0^{-1}(1) = p_1^{-1}(\#\AAA)$ and $z= p_0^{-1}(\#\AAA) = p_1^{-1}(1)$.
			\begin{enumerate}
				\item If $\qq$ and $\hat{\qq}$ are the results of an (inner) $\{c,d\}$-switch, $c,d\in\AAA\setminus\{a,z\}$, on $\pp$ and $\hat{\pp}$ respectively, then $\hat{\qq} = \Ext{x}{b_0}{b_1}{\qq}$.
				\item If $\rr$ and $\hat{\rr}$ are the results of an (outer) $\{a,c\}$-switch, $c\in\AAA\setminus\{a,z\}$, on $\pp$ and $\hat{\pp}$ respectively, then 
					\begin{enumerate}
						\item $\hat{\rr} = \Ext{x}{z}{b_1}{\rr}$ if $c=b_0$,
						\item $\hat{\rr} = \Ext{x}{a}{b_1}{\rr}$ if $b_0= z$, and
						\item $\hat{\rr} = \Ext{x}{b_0}{b_1}{\rr}$ otherwise.
					\end{enumerate}
				\item If $\uu$ and $\hat{\uu}$ are the results of an (outer) $\{c,z\}$-switch, $c\in\AAA\setminus\{a,z\}$, on $\pp$ and $\hat{\pp}$ respectively, then
					\begin{enumerate}
						\item $\hat{\uu} = \Ext{x}{b_0}{a}{\uu}$ if $c=b_1$,
						\item $\hat{\uu} = \Ext{x}{b_0}{z}{\uu}$ if $b_1=a$, and
						\item $\hat{\uu} = \Ext{x}{b_0}{b_1}{\uu}$ otherwise.
					\end{enumerate}
			\end{enumerate}
		\end{lem}
		
	\section{Reduction to Block Permutations}

	    In this section, we will present results that show convenient forms of pairs occur in every Ruazy Class and Extended Rauzy Class. Therefore, by Lemma \ref{LemNamingsClass} we may restrict our arguments to such pairs in order to prove our main result.

	    For the remainder of this paper, a \emph{block} in standard $\pp = (p_0,p_1) \in \irr{\AAA}$ is any set $\BBB\subseteq \AAA$ such that $p_0(\BBB) = p_1(\BBB) = \{k, \dots, k+\#\BBB -1\}$ for some $1<k<\#\AAA - \#\BBB-1$. Note that this does not necessarily imply that each letters in $\BBB$ appears in the same position on both rows of $\pp$. We will also use the term \emph{block} to refer to the pattern that occurs in $\pp$ on letters in $\BBB$.
  
	    \subsection{Special Blocks}

		We say that a standard $\pp\in\irr{\AAA}$ is \emph{composed of blocks} when $\AAA = \BBB_0 \cup \BBB_1\cup \dots \cup \BBB_k$ where $\BBB_0=\{p_0^{-1}(1),p_1^{-1}(1)\}$ are the first and last letters of each row of $\pp$, all $\BBB_j$'s are pair-wise disjoint and for each $j\geq 1$ the set $\BBB_j$ forms a block in $\pp$. We refer the reader to \cite{cFick2014} and \cite{cFick2012} for proofs of the following.

		\begin{prop}\label{PropForms}
		      For every $\pp\in\irr{\AAA}$ there exists $\qq$ in the Extended Rauzy Class of $\pp$ so that $\qq$ is composed of blocks, and each block is of one of the following forms:
				  \begin{enumerate}
				  \setcounter{enumi}{-1}
				    \item An empty block.
				    \item\label{FormOR} Of the form
					    $$ \mtrx{\LL{a_1}{a_n}\; \LL{a_2}{a_{n-1}} \;\LL{ \dots}{\dots} \;\LL{ a_n}{a_1}}$$
					  for some $n\geq 2$.
				    \item\label{FormTwos} Of the form
					    $$ \mtrx{\LL{a_1}{b_1}\;\LL{b_1}{a_1}\;\LL{a_2}{b_2}\;\LL{b_2}{a_2}\;\LL{\dots}{\dots}\;\LL{a_n}{b_n}\;\LL{b_n}{a_n}}$$
					  for some $n\geq 1$.
					 \item\label{FormFourAndTwos} Of the form
					 	$$ \mtrx{\LL{a_1}{a_4}\;\LL{a_2}{a_3}\;\LL{a_3}{a_2}\;\LL{a_4}{a_1}~
					 		\LL{b_1}{c_1}\;\LL{c_1}{b_1}\;\LL{b_2}{c_2}\;\LL{c_2}{b_2}\;\LL{\dots}{\dots}\;\LL{b_n}{c_n}\;\LL{c_n}{b_n}}$$
					 	for some $n\geq 0$.
					 \item\label{FormThreeAndTwos} Of the form
					 	$$ \mtrx{\LL{a_1}{b_1}\;\LL{b_1}{a_1}\;\LL{a_2}{b_2}\;\LL{b_2}{a_2}\;\LL{\dots}{\dots}\;\LL{a_m}{b_m}\;\LL{b_m}{a_m}~\LL{c}{e}\;\LL{d}{d}\;\LL{e}{c}~\LL{f_1}{g_1}\;\LL{g_1}{f_1}\;\LL{f_2}{g_2}\;\LL{g_2}{f_2}\;\LL{\dots}{\dots}\;\LL{f_n}{g_n}\;\LL{g_n}{f_n}}$$
					 	for some $m,n\geq 0$.
				  \end{enumerate}
		\end{prop}

		Pairs of this form will be very helpful for the proofs that follow, and we will already see one immediate benefit: for such $\pp$ such that $\Spinof{\pp}$ is well-defined, this value may be calculated quite easily. The following Lemma and Corollary are direct results from Section 2.1 of \cite{cFick2014}.

		\begin{lem}\label{LemSpinTwoBlocks}
			    Let $\pp = \mtrx{\LL{a}{z}~B_1~\LL{s}{s}~B_2~\LL{z}{a}}, \qq = \mtrx{\LL{a}{z}~B_1~\LL{z}{a}}\mbox{ and } \rr = \mtrx{\LL{a}{z}~B_2~\LL{z}{a}}.$
			    If $\Spinof{\qq}$ and $\Spinof{\rr}$ are well defined then so is $\Spinof{\pp}$ and
				$$ \Spinof{\pp} + 1 = \Spinof{\qq} + \Spinof{\rr}~\mathrm{mod }2.$$
		\end{lem}

		\begin{cor}\label{CorSpinBlocks}
			If standard $\pp\in\irr{\AAA}$ is composed of exactly $N_1$ empty blocks, $N_2$ blocks of Form \ref{FormTwos}, $N_3$ blocks of Form \ref{FormFourAndTwos} and $N_4$ blocks of Form \ref{FormOR} with $n=5$, then $\Spinof{\pp}$ exists and is equal to
				\begin{equation}
					\Spinof{\pp} = 1 + N_3 + N_4~\mathrm{mod }2.
				\end{equation}
		\end{cor}

		The next proposition is a result in \cite{cFick2012}, and it is essential to the main result of that work. This is a combinatorial statement of one implication of the classification theorems from \cite{cKonZor2003}.

		\begin{prop}\label{PropTypes}
			For every $\pp\in\irr{\AAA}$ there exists $\qq$ in the Extended Rauzy Class of $\pp$ so that $\qq$ is composed of blocks and is of one of the following forms:
				\begin{enumerate}
					\item\label{TypeHyp} \textbf{Hyperelliptic:} $\qq$ is composed of at most one block of Form \ref{FormOR} and the remaining blocks are empty.
					\item\label{TypeEvenOdd} \textbf{Odd Cycles, Odd Spin:} $\qq$ is composed of blocks of Form \ref{FormTwos} and empty blocks.
					\item\label{TypeTwosEven} \textbf{3/1 Cycles, Even Spin:} $\qq$ is composed of exactly one block of Form \ref{FormOR} with $n=5$ and the rest are of Form \ref{FormTwos} or empty.
					\item\label{TypeEvenEven} \textbf{Odd Cyclyes, Even Spin:} $\qq$ is composed of exactly one block of Form \ref{FormFourAndTwos} and the rest are of Form \ref{FormTwos} or empty.
					\item\label{TypeOdd} \textbf{Even Cycles:} $\qq$ is composed of blocks that are empty, Form \ref{FormTwos} and Form \ref{FormThreeAndTwos}, and at least one block is Form \ref{FormThreeAndTwos}.
				\end{enumerate}
		\end{prop}
		
		\begin{cor}\label{CorUniqueType}
			If $\pp$ is of one of the Types in Proposition \ref{PropTypes} and $\qq$ is in the same Extended Rauzy Class as $\pp$, then there exists $\rr$ in the Rauzy Class of $\qq$ that is the same Type as $\pp$.
		\end{cor}
		
		\begin{proof}
			The reader may directly verify that $\pp$ of Type \ref{TypeHyp} can only have Type \ref{TypeHyp} pairs in its Extended Rauzy Class. Otherwise $\Pof{\pp}$ and $\Spinof{\pp}$ (which is well defined for $\pp$ of Types \ref{TypeEvenOdd} -- \ref{TypeEvenEven}) uniquely determine a Type. By Corollary \ref{CorPofInvariant} and Proposition \ref{prop_spin}, any $\qq$ in the Extended Rauzy Class of $\pp$ must satisfy $\Pof{\qq}=\Pof{\pp}$ and $\Spinof{\qq}=\Spinof{\pp}$. So if $\qq$ is of a Type, it must be the same Type as $\pp$.
		\end{proof}

		One final result also comes from \cite{cFick2012}.
		\begin{prop}\label{PropSufficient}
			If $\pp\in\irr{\AAA}$ is of one of the Types in Proposition \ref{PropTypes}, $\tilde{\qq}\in \irr{\AAA'}$ is the same Type as $\pp$ and $\Pof{\pp} = \Pof{\tilde{\qq}}$, then there exists $\qq$ in the Extended Rauzy Class of $\pp$ such that $\qq = \tilde{\qq}\circ \nu'$, where $\nu':\AAA \to \AAA'$ is a bijection.
		\end{prop}

	  \subsection{Prefix Insertions: Combining Blocks}

		As indicated in Lemma \ref{LemSwitchAndInsert}, we may apply a fixed Prefix Insertion rule to all standard elements of a Rauzy Class, as inner switch moves commute with insertions. Therefore, if $\hat{\pp}= \Ext{x}{b_0}{b_1}{\pp}$ has a well-defined value $\Spinof{\hat{\pp}}$, then so does $\hat{\qq} = \Ext{x}{b_0}{b_1}{\qq}$ for every $\qq\in \ClassLab{\pp}$ and the spin values all agree. This argument was used in \cite{cBoi2013}. However, when considering Extended Rauzy Classes, this is no longer the case. In order to compensate, we must consider \emph{all Prefix Insertions} that satisfy certain conditions. We then show a parity argument: two such Prefix Insertions on $\pp$ with resulting pairs $\hat{\pp}$ and $\hat{\pp}'$ satisfy $\Spinof{\hat{\pp}} = \Spinof{\hat{\pp}'}$ if and only if $\Nof{\hat{\pp}}$ and $\Nof{\hat{\pp}'}$ are related by an \emph{even} permutation.
	  
%	  	TODO: Explain how this relates to \cite{cBoi2013}. Key issue: while a Prefix Insertion rule may be fixed throughout a Rauzy Class, this is no longer the case for Extended Rauzy Classes. Instead, we will consider \emph{all Prefix Insertions} that satisfy certain conditions. We then show a parity argument, two such Prefix Insertions on a fixed $\pp$ yield new pairs $\hat{\pp}$ and $\hat{\pp}'$ that satisfy $\Spinof{\hat{\pp}} = \Spinof{\hat{\pp}'}$ if and only if $\Nof{\hat{\pp}}$ and $\Nof{\hat{\pp}'}$ are related by an \emph{even} permutation.
	  	
	  	We begin with two introductory lemmas. In the first lemma,  $\pp\in\irr{\AAA}$ is composed of two blocks that each represent an odd length cycle of $\Pof{\pp}$. The Prefix Insertion rule will result in $\hat{\pp}$ with one odd length cycle in $\Pof{\hat{\pp}}$ that satisfies $\Pof{\hat{\pp}} = \Pof{\pp}$, regardless of which insertion rule was chosen. The second lemma considers $\pp$ so that $\Pof{\pp}$ contains exactly two cycles, each of even length. Again, the resulting $\hat{\pp}$ has exactly one cycle of odd length. However, the spin may vary depending on choice of insertion rule.

		\begin{lem}\label{LemJoining1}
		      Let $\pp\in\irr{\AAA}$ be of the form
			  $$ \pp = \mtrx{\LL{a}{z}\; B_1 \LL{s}{s}\; B_2\; \LL{z}{a}}$$
		      where $B_j$ is Form \ref{FormTwos} or \ref{FormFourAndTwos}. Also, let $\BBB_1 = B_1 \cup\{s\}$ and $\BBB_2 = B_2 \cup \{z\}$ denote the subalphabets of $\AAA$ that consist of the letters in block $B_1$ with $s$ and $B_2$ with $z$ respectively. If $$\hat{\pp} = \Ext{x}{b}{c}{\pp}$$
		      where $x\notin\AAA$, $b\in \BBB_2$ and $c\in\BBB_1$, then $\Spinof{\pp} = \Spinof{\hat{\pp}}$.
		      Furthermore, if $\hat{\pp}' = \Ext{x}{b'}{c'}{\pp}$ for some other choice of $b'\in\BBB_2$ and $c'\in\BBB_1$, then
%			    $ \Nof{\hat{\pp}} = \Nof{\hat{\pp}'}.$
			    $$ \Nof{\hat{\pp}} = \Nof{\hat{\pp}'}* \nu$$
		      where $\nu \in \Zof{\pp}$.
		\end{lem}

		\begin{proof}
		      We will show the proof for $B_1$ and $B_2$ both of Form \ref{FormTwos}. The remaining cases are similar. $\Spinof{\pp}=1$ by Corollary \ref{CorSpinBlocks}. Let $\pp$ be of the form
			    $$ \pp = \mtrx{\LL{a}{z}~\LL{b_1}{c_1}\;\LL{c_1}{b_1}\;\LL{\cdots}{\cdots}\;\LL{b_m}{c_m}\;\LL{c_m}{b_m}~\LL{s}{s}~
				    \LL{d_1}{e_1}\;\LL{e_1}{d_1}\;\LL{\cdots}{\cdots}\;\LL{d_n}{e_n}\;\LL{e_n}{d_n}~\LL{z}{a}}.$$
		      We claim in this case that for any choice of $b\in \BBB_2$ and $c\in \BBB_1$ the pair $\hat{\pp} = \Ext{x}{b}{c}{\pp}$ satisfies $\Spinof{\hat{\pp}}=1$, finishing the proof.

		      We first see directly that if $b = d_1$ and $c = s$, then
			      $$ \hat{\pp} = \mtrx{\LL{a}{z}~\LL{b_1}{c_1}\;\LL{c_1}{b_1}\;\LL{\cdots}{\cdots}\;\LL{b_m}{c_m}\;\LL{c_m}{b_m}~\LL{s}{x}\;\LL{x}{s}~
				    \LL{d_1}{e_1}\;\LL{e_1}{d_1}\;\LL{\cdots}{\cdots}\;\LL{d_n}{e_n}\;\LL{e_n}{d_n}~\LL{z}{a}}$$
		      satsfies the claim. If $c\neq s$, then we may perform a $\{c,s\}$-switch on $\pp$ to get $\pp'$ of the form
			      $$ \pp' = \mtrx{\LL{a}{z} \; B'_1~ \LL{c}{c} ~\LL{d_1}{e_1}\;\LL{e_1}{d_1}\;\LL{\cdots}{\cdots}\;\LL{d_n}{e_n}\;\LL{e_n}{d_n}~\LL{z}{a}}$$
		      where $B_1'$ is a block with letters $\BBB_1\setminus\{c\}$. We verify directly that $B_1'$ is not of Form \ref{FormOR} if $m>1$. Let
			      $$ \rr = \mtrx{\LL{a}{z}\;B_1'~\LL{z}{a}}$$
		      and note that $\pp'$ is just the result of a Prefix Insertion rule on $\rr$. By noting that $\Sof{\rr}$ is just one cycle, \ref{PropTypes} tells us that there exists $\rr' \in \ClassLab{\rr}$ of either Type \ref{TypeEvenOdd} or \ref{TypeEvenEven}. This means that $\rr'$ is composed of exactly one block $B_1''$ of Form \ref{FormTwos} or (resp.) \ref{FormFourAndTwos}. Because insertion rules commute with inner switch moves, there exists
			    $$ \pp'' = \mtrx{\LL{a}{z} \; B''_1 ~ \LL{s}{s} ~\LL{d_1}{e_1}\;\LL{e_1}{d_1}\;\LL{\cdots}{\cdots}\;\LL{d_n}{e_n}\;\LL{e_n}{d_n}~\LL{z}{a}}$$
		      such that is the insertion on $\rr'$ (by the same rule as from $\rr$ to $\pp'$) and so $\pp''\in \ClassLab{\pp}$. By Proposition \ref{prop_spin}, $\Spinof{\pp''} = \Spinof{\pp}$ and so $B_1''$ must be of Form \ref{FormTwos}. Therefore, $\pp''$ is of the same form as $\pp$.

		      If $b=d_1$ in this case, then $\hat{\pp}'' = \Ext{x}{b}{c}{\pp''}$ belongs to $\ClassLab{\hat{\pp}}$ and $\Spinof{\hat{\pp}''} = 1$. Again by Proposition \ref{prop_spin}, $\Spinof{\pp} = 1$. If $b\neq d_1$, then we may perform an \emph{outer} switch followed by inner switch moves to place $b$ into the position of $d_1$ in the original pair $\pp$ (we will call this resulting pair $\pp'''$ and $\omega$ the path from $\pp$ to $\pp'''$). This outer switch involves $z$, so by Lemma \ref{LemSwitchAndInsert} $\hat{\pp}'''$ is result of $\omega$ on $\hat{\pp}$ where $\hat{\pp}''' = \Ext{x}{b}{c}{\pp'''}$. Just as in our simplest case, $\Spinof{\hat{\pp}'''} = 1$ as the pair is Form \ref{FormTwos} and so $\Spinof{\hat{\pp}} = 1$ by Proposition \ref{prop_spin}.

		      The final statement remains to be proved. We note that if $\hat{\pp} = \Ext{x}{b}{c}{\pp}$ and $\hat{\pp}' = \Ext{x}{b'}{c'}{\pp}$, then there are paths $\omega$ and $\omega'$ so that $\hat{\rr} = \Ext{x}{b}{c}{\omega\pp}$ and $\hat{\rr}'=\Ext{x}{b'}{c'}{\omega'\pp}$ are Form \ref{FormTwos} and are of the form:
				$$ \hat{\rr} = \mtrx{\LL{a}{z'}~B_1~\LL{s'}{x}\;\LL{x}{s'}~B_2~\LL{z'}{a}} \mbox{ and }
				      \hat{\rr}' =\mtrx{\LL{a}{z''}~B'_1~\LL{s''}{x}\;\LL{x}{s''}~B'_2~\LL{z''}{a}}.$$
		      If $\rr = \omega\pp$ and $\rr' = \omega'\pp$, then $\rr = \rr' \circ \nu$, where $\nu \in \Zof{\rr'} = \Zof{\pp}$ by Lemma \ref{LemQuotientOverZof}, as $\mu_{\omega}$ and $\mu_{\omega'}$ from \ref{DefMuOmega} are trivial. By inspection, we see that $\hat{\rr} = \hat{\rr}' \circ \nu$. We then conclude that
			      $$ \Zof{\hat{\pp}} = \Zof{\hat{\rr}} = \Zof{\hat{\rr}'}*\nu =\Zof{\hat{\pp}'}*\nu.$$

		\end{proof}

		\begin{lem}\label{LemJoining2}
		      Let $\pp\in\irr{\AAA}$ be of the form
			    $$ \pp = \mtrx{\LL{a}{z}\;\LL{b_1}{c_1}\;\LL{c_1}{b_1}\;\LL{\dots}{\dots}\;\LL{b_m}{c_m}\;\LL{c_m}{b_m}\;\LL{d}{f}\;\LL{e}{e}\;\LL{f}{d}\;
					  \LL{g_1}{h_1}\;\LL{h_1}{g_1}\;\LL{\dots}{\dots}\;\LL{g_n}{h_n}\;\LL{h_n}{g_n}\;\LL{z}{a}}$$
		      or equivalently $\pp$ is composed of one block of Form \ref{FormThreeAndTwos}. Also, let $$\CCC^- := \{b_1,\dots,b_m,c_1,\dots,c_m,d,f\}\mbox{ and }\CCC^+ :=\{e,g_1,\dots,g_n,h_1,\dots,h_n,z\}$$
		      denote the \emph{left} and \emph{right} letters of the block respectively. If
			      $$ \hat{\pp} = \Ext{x}{b}{c}{\pp}$$
		      where $x\notin\AAA$, $b\in \CCC^+$ and $c\in \CCC^-$, then $\Spinof{\hat{\pp}}$ is well defined.

		      Furthermore, if $\hat{\pp}' = \Ext{x}{b'}{c'}{\pp}$ for some other choice of $b'\in \CCC^+$,$c'\in\CCC^-$, then
			$$ \Nof{\hat{\pp}} = \Nof{\hat{\pp}'}*\nu$$
		      where $\nu\in\Zof{\pp}$. Also, $\Spinof{\hat{\pp}} = \Spinof{\hat{\pp}'}$ if and only if $\nu$ is even.
		\end{lem}

		\begin{proof}[Sketch of Proof]
			Our argument will follow a similar strategy to that of the previous lemma, so we will only explain the differences here.
			In the previous proof, we reduced to one simple base case, while there are two different cases here.

			If $\hat{\pp} = \Ext{x}{e}{d}{\pp}$, then
			    $$ \hat{\pp} = \mtrx{\LL{a}{z}\;\LL{b_1}{c_1}\;\LL{c_1}{b_1}\;\LL{\dots}{\dots}\;\LL{b_m}{c_m}\;\LL{c_m}{b_m}\;\LL{d}{f}\;\LL{x}{e}\;\LL{e}{x}\;\LL{f}{d}\;
					  \LL{g_1}{h_1}\;\LL{h_1}{g_1}\;\LL{\dots}{\dots}\;\LL{g_n}{h_n}\;\LL{h_n}{g_n}\;\LL{z}{a}}$$
			and $\Spinof{\hat{\pp}} = 0$ by Corollary \ref{CorSpinBlocks}. If instead $\hat{\pp} = \Ext{x}{e}{f}{\pp}$, then
			    $$ \hat{\pp} = \mtrx{\LL{a}{z}\;\LL{b_1}{c_1}\;\LL{c_1}{b_1}\;\LL{\dots}{\dots}\;\LL{b_m}{c_m}\;\LL{c_m}{b_m}\;\LL{d}{x}\;\LL{x}{f}\;\LL{e}{e}\;\LL{f}{d}\;
					  \LL{g_1}{h_1}\;\LL{h_1}{g_1}\;\LL{\dots}{\dots}\;\LL{g_n}{h_n}\;\LL{h_n}{g_n}\;\LL{z}{a}}.$$
			    We may directly verify that $\Spinof{\hat{\pp}} = 1$ in this case. To do this, we use Lemma \ref{LemSpinTwoBlocks} after determining that the pair $\rr= \mtrx{\LL{a}{z}\;\LL{d}{x}\;\LL{x}{f}\;\LL{e}{e}\;\LL{f}{d}\;\LL{z}{a}}$ is in the same labeled Rauzy Class as pair $\mtrx{\LL{a}{z}\;\LL{f}{d}\;\LL{d}{f}\;\LL{e}{x}\;\LL{x}{d}\;\LL{z}{a}}$ and therefore $\Spinof{\rr} = 1$.

			    We again will then consider a general $\hat{\pp}$ and use inner switch moves and possibly an outer switch move (that involves $z$ and not $a$) to transform any other insertion into one of the base cases. Proposition \ref{prop_spin} then tells us that the base case is unique.
		\end{proof}

		We now arrive at the main lemma of this section. We begin with a special form of $\pp$ and define a class of Prefix Insertion rules so that:
			\begin{itemize}
				\item Each pair of even length cycles is combined into an odd length cycle, and
				\item once all cycles are of odd length, these are all combined into one odd length cycle. 
			\end{itemize}
		Because $\pp$ is fixed in the lemma, the choices of which cycles are paired are completely determined.
		\begin{lem}\label{LemCombinedInserts}
		      Let $\pp\in\irr{\AAA}$ be composed of blocks and of the form
			   $$ \pp = \mtrx{\LL{a}{t_n}\;B_1\; \LL{s_1}{s_1}\; B_2\;\LL{s_2}{s_2}\;\dots\; B_m\;\LL{s_m}{s_m}\;C_1\;\LL{t_1}{t_1}\;\dots\; \LL{t_{n-1}}{t_{n-1}}\; C_n \;\LL{t_n}{a}},$$
		      $m,n\geq 0$, $B_j$ is of Form \ref{FormTwos} or \ref{FormFourAndTwos} for $1\leq j \leq m$ and $C_j$ is of Form \ref{FormThreeAndTwos} for $1\leq j \leq n$. Let $\BBB_j = B_j\cup \{s_j\}$ be the subalphabets of $\AAA$ related to block $B_j$, $1\leq j \leq m$ and $\CCC_j^{-}$ (resp. $\CCC_j^+$) be the left (resp. right) subalphabet related to block $C_j$, $1\leq j \leq n$. Here $t_j$ replaces the role of $z$ in $\CCC_j^+$.
		      
		      Consider the following choices of letters:
			  \begin{enumerate}
			   \item $x_1,\dots,x_{m+2n-1}\notin \AAA$, $x_i\neq x_j$ for $i\neq j$ and
			   \item $b_1,\dots,b_{m+2n-1},c_1,\dots,c_{m+2n-1}\in\AAA$ so that:
				\begin{enumerate}
				    \item $b_j \in \CCC^+_j$, $c_j\in \CCC_j^-$, $1\leq j \leq n$,
				    \item $b_{n+j} \in \CCC^-_{j+1}$, $c_{n+j}\in \CCC^+_{j}$, $1\leq j < n$
				    \item \begin{enumerate}
						  \item if $n>0$, $b_{2n-1+j} \in \BBB_{j+1}$, $c_{2n-1+j} \in \BBB_j$, $1\leq j <m$,
						\item if $n=0$, $b_{j} \in \BBB_{j+1}$, $c_{j} \in \BBB_j$, $1\leq j <m$, and
					  \end{enumerate}
				    \item $b_{m+2n-1}\in \CCC_1^-$, $c_{m+2n-1}\in\BBB_m$ if $m>0$.
				\end{enumerate}
			  \end{enumerate}

		      If $\hat{\pp} = \Ext{\vec{x}}{\vec{b}}{\vec{c}}{\pp}$ then $\Spinof{\hat{\pp}}$ is well defined. If $\vec{b}',\vec{c}'\in\AAA^{m+2n-1}$ also satisfy the above conditions and $\hat{\pp}' = \Ext{\vec{x}}{\vec{b}'}{\vec{c}'}{\pp}$, then
			      $$ \Nof{\hat{\pp}'} = \Nof{\hat{\pp}}* \nu$$
		      where $\nu\in\perm{\AAA\cup\{x_1,\dots,x_{n+2m-1}\}} $. Furthermore, $\Spinof{\hat{\pp}'} = \Spinof{\hat{\pp}}$ if and only if $\nu$ is even.
		\end{lem}

		\begin{proof}

		    By Lemma \ref{LemPrefixes}, we may choose to apply our insertions in any order. If $n=0$, we may apply Lemma \ref{LemJoining1} to join $B_1$ to $B_2$. The result is a pair of the same form with $m-1$ $B_j$ blocks. This process terminates in $m-1$ steps and the resulting combined $\nu$ is even, as at each step all cycles of $\Pof{\pp}$ are even permutations. Furthermore, $\Spinof{\pp} = \Spinof{\hat{\pp}}$ always.

		    If $n>0$, we may first apply Lemma \ref{LemJoining2} to transform each $C_j$ block into Forms \ref{FormTwos} or \ref{FormFourAndTwos}. Each step will result in $\nu=\nu_1\dots \nu_n$. The remaining insertions will neither alter $\Spinof{\hat{\pp}}$ nor change whether $\nu$ is even/odd (as discussed in the first paragraph). If we arrive at $\hat{\pp}' = \Ext{\vec{x}}{\vec{b}'}{\vec{c}'}{\pp}$, then we again have $\nu'$ from transforming the $C_j$ into Forms \ref{FormTwos} or \ref{FormFourAndTwos}. We check by induction that $\nu^{-1}\nu'$ is even if and only if $\Spinof{\hat{\pp}} = \Spinof{\hat{\pp}'}$.
		 
		\end{proof}
		
	\section{Containing the Alternating Subgroup}

	  \subsection{Single Block Permutations}
	    
	    In this section, we will show that $\Eperm{\AAA}\subseteq \Namings{\pp}$ for any $\pp\in\irr{\AAA}$ of the form
	    	\begin{equation}
	    		\pp = \mtrx{\LL{a}{z}\;B\;\LL{z}{a}}
	    	\end{equation}
		where $B$ is a single non-empty block of one of the Forms in Proposition \ref{PropForms}.
		
		However, we first begin with a result concerning standard $\pp\in\irr{\AAA}$ with $\Pof{\pp} = \{1^n\} = \{1,\dots,1\}$ consisting the value $1$ repeated $n$ times. In particular $n=k+1$ for $\pp$ in the following.
		
			\begin{lem}[All empty blocks]\label{LemEmpty}
				If $\pp=\irr{\AAA}$ is of the form
					$$ \pp = \mtrx{\LL{a}{z}\;\LL{s_1}{s_1}\;\LL{s_2}{s_2}\;\LL{\dots}{\dots}\;\LL{s_k}{s_k}\;\LL{z}{a}}$$
				for $k\geq 2$, then $\perm{\{s_1,s_2,\dots,s_k\}}\subset \Namings{\pp}$.
			\end{lem}
			
			\begin{proof}
				We will show that $(s_1,s_j)\in\Namings{\pp}$ for $2\leq j\leq k$. If we act on $\pp$ by an $\{s_1,s_2\}$-switch, then we arrive at
					$$ \mtrx{\LL{a}{z}\;\LL{s_2}{s_2}\;\LL{s_1}{s_1}\;\LL{\dots}{\dots}\;\LL{s_k}{s_k}\;\LL{z}{a}}$$
				and so $(s_1,s_2)\in\Namings{\pp}$. For $j>2$, perform an $\{s_1,s_j\}$-switch followed by an $\{s_{j-1},s_j\}$-switch on $\pp$. The resulting permutation is given by the renaming $(s_1,s_j)$ on $\pp$.
			\end{proof}
		
			\begin{lem}[$B$ is Form \ref{FormOR}]\label{LemOR}
				If $\AAA = \{1,2,\dots,N\}$ and $\pp\in\irr{\AAA}$ is
					$$ \pp = \mtrx{\LL{1}{N}\;\LL{2}{N-1}\;\LL{\dots}{\dots}\;\LL{N-1}{2}\;\LL{N}{1}}$$
				or $p_0(n) + p_1(n) = N+1$, for each $n\in \AAA$ then
					$$ \Namings{\pp} = \RHScase{\Eperm{\AAA}, & \mbox{if }N\mbox{ is even},\\ \perm{\AAA}, & \mbox{if }N\mbox{ is odd}.}$$
			\end{lem}
					
			\begin{proof}
				By definition, an inner switch is not possible on $\pp$. Furthermore, for any outer switch $\omega$, the resulting $\tilde{\pp}$ satisfies $\tilde{\pp}\simperm\pp$. Therefore every switch path from $\pp$ must be composed of outer switches and
					\begin{equation}
						\Namings{\pp} = <\alpha_{\{1,2\}},\dots,\alpha_{\{1,j\}},\dots,\alpha_{\{1,N-1\}},\alpha_{\{2,N\}},\dots,\alpha_{\{j,N\}},\dots, \alpha_{\{N-1,N\}}>,
					\end{equation}
				where $\alpha_{\omega}$ satisfies $\tilde{\pp} = \pp\circ\alpha_\omega$ and $\tilde{\pp}$ is $\omega\pp$ for switch move $\omega$.
				
				If $\omega$ is the $\{1,2\}$-switch move, then the resulting $\qq\in\irr{\AAA}$ is
					$$ \qq = \mtrx{\LL{2}{N}\;\LL{\dots}{1}\;\LL{N-1}{N-1}\;\LL{1}{\dots}\;\LL{N}{2}}$$
				and so $\nu_1 := \alpha_{\{1,2\}}= (N-1,N-2,\dots,2,1)$. We may see that if $\omega = \{1,\ell\}$ for $2\leq \ell \leq N-1$, then $\alpha_{\{1,\ell\}} = \nu_1^{\ell-1}$.
				
				If $\omega$ is the $\{N-1,N\}$-switch move, then the resulting $\rr\in\irr{\AAA}$ is
					$$ \rr = \mtrx{\LL{1}{N-1}\;\LL{N}{\dots}\;\LL{2}{2}\;\LL{\dots}{N}\;\LL{N-1}{1}}$$
				and $\nu_2 := \alpha_{\{N-1,N\}} = (2,3,\dots,N)$. We may also verify that $\alpha_{\{\ell,N\}} = \nu_2^{N-\ell}$ for $2\leq \ell \leq N-1$.
				
				We conclude by Lemma \ref{LemEvenOddGenerators} that
					$$ \Namings{\pp} = <\nu_1,\nu_2> = <\nu_1^{-1},\nu_2> = \RHScase{\Eperm{\AAA}, & \mbox{if }N\mbox{ is even},\\ \perm{\AAA}, & \mbox{if }N\mbox{ is odd},}$$
				as desired.
			\end{proof}
			
			\begin{lem}[$B$ is Form \ref{FormTwos} or \ref{FormFourAndTwos}]\label{LemEvenSpin}
				If $\pp\in\irr{\AAA}$ is composed of one block of Form \ref{FormTwos} or \ref{FormFourAndTwos}, then $\Eperm{\AAA}\subseteq\Namings{\pp}$.
			\end{lem}
			
			\begin{proof}
				We will show the proof for $\pp$ of the form
					$$ \pp = \mtrx{\LL{a}{z}\;\LL{b_1}{c_1}\;\LL{c_1}{b_1}\;\LL{\dots}{\dots}\;\LL{b_n}{c_n}\;\LL{c_n}{b_n}\;\LL{a}{z}}$$
				with $\AAA = \{a,b_1,\dots,b_n,c_1,\dots,c_n,z\}$ as the proof for a block of Form \ref{FormFourAndTwos} is similar. Note that $\pp$ is Type \ref{TypeEvenOdd} from Proposition \ref{PropTypes} and
					$$ \Sof{\pp} = (a,c_n,c_{n-1},\dots, c_1,b_1,b_2,\dots,b_n,z).$$
				Let $\qq$ be the result of an $\{a,c_n\}$-switch on $\pp$. Then by direct calculation, we see that
					$$ \Sof{\qq} = (c_n,c_{n-1},\dots,c_1,b_1,b_2,\dots,b_n,a,z).$$
				By Corollary \ref{CorUniqueType}, there exists $\tilde{\qq}$ in the Rauzy Class of $\qq$ that is Type \ref{TypeEvenOdd}. Because $\Sof{\tilde{\qq}} = \Sof{\qq}$, we know that $\tilde{\qq}$ is composed of exactly one block of Form \ref{FormTwos} and is therefore
					$$ \tilde{\qq} = \mtrx{\LL{c_n}{z}\;\LL{b_2}{b_1}\;\LL{b_1}{b_2}\;\LL{b_3}{c_1}\;\LL{c_1}{b_3}\;\LL{\dots}{\dots}\;\LL{b_n}{c_{n-2}}\;\LL{c_{n-2}}{b_n}\;\LL{a}{c_{n-1}}\;\LL{c_{n-1}}{a}\;\LL{z}{c-n}}$$
				Therefore $\tilde{\qq} = \pp \circ \nu_1$ for $\nu_1 = (a,b_n,b_{n-1},\dots,b_1,c_1,c_2\dots,c_n)\in\Namings{\pp}$.
				
				If we act on $\pp$ by a $\{b_n,z\}$-switch move and call the resulting permutation $\rr$, then
					$$ \Sof{\rr} = (a,z,c_n,c_{n-1},\dots,c_1,b_1,b_2,\dots,b_n).$$
				Again by Corollary \ref{CorUniqueType}, $\tilde{\rr}$ exists in the Extended Rauzy Class of $\pp$ and is
					$$ \tilde{\rr} = \mtrx{\LL{a}{b_n}\;\LL{c_1}{c_2}\;\LL{c_2}{c_1}\;\LL{b_1}{c_3}\;\LL{c_3}{b_1}\;\LL{\dots}{\dots}\;\LL{b_{n-2}}{c_n}\;\LL{c_n}{b_{n-2}}\;\LL{b_{n-1}}{z}\;\LL{z}{b_{n-1}}\;\LL{b_n}{a}}.$$
				We see that $\tilde{\rr} = \pp \circ \nu_2$ for $\nu_2=(z,c_n,c_{n-1},\dots,c_1,b_1,b_2,\dots,b_n)\in\Namings{\pp}$.
				
				We conclude by Lemma \ref{LemEvenOddGenerators} because $\Namings{\pp}$ contains
					$$ <\nu_1,\nu_2> = <\nu_1^{-1},\nu_2> = \Eperm{\AAA}$$
				as $\#\AAA = 2n +2$.
			\end{proof}
			
		\begin{lem}[$B$ is Form \ref{FormThreeAndTwos}]\label{LemEvenCycles}
			If $\pp\in\irr{\AAA}$ is composed of one block of Form \ref{FormThreeAndTwos} then $\Eperm{\AAA}\subseteq\Namings{\pp}$. Furthermore, if $m=n$, where $m,n\geq 0$ come from the definition of Form \ref{FormThreeAndTwos}, then $\Namings{\pp} = \perm{\AAA}$.
		\end{lem}
		
		\begin{proof}
			We will express $\pp$ as
				$$ \pp = \mtrx{\LL{a}{z}\;\LL{b_1}{c_1}\;\LL{c_1}{b_1}\;\LL{\dots}{\dots}\;\LL{b_m}{c_m}\;\LL{c_m}{b_m}\;\LL{d\;e\;f}{f\;e\;d}\;\LL{g_1}{h_1}\;\LL{h_1}{g_1}\;\LL{\dots}{\dots}\;\LL{g_n}{h_n}\;\LL{h_n}{g_n}\;\LL{z}{a}}$$
			with $\Sof{\pp} = (a,h_n,\dots,h_1,e,g_1,\dots,g_n,z)(b_1,\dots,b_m,d,f,c_m,\dots,c_1)$.			
			
			Let $\BBB = \{a,e,g_1,\dots,g_n,h_1,\dots,h_n,z\}$. We will show that $\Eperm{\BBB}\subseteq\Namings{\pp}$. By performing an $\{a,b_1\}$-switch, we arrive at a new permutation $\tilde{\qq}$ with $$\Sof{\tilde{\qq}} = (b_1,\dots,b_m,d,f,c_m,\dots,c_1,z)(a,h_n,\dots,h_1,e,g_1,\dots,g_n)$$ and $\Xof{\qq}=b_1$. By Corollary \ref{CorUniqueType}, there exists $\qq$ in the Rauzy Class of $\tilde{\qq}$ of Form \ref{FormThreeAndTwos} with the role of $m$ and $n$ reversed. The following proof would then show that $\Eperm{\BBB'} \subseteq \Namings{\qq}$, where $\BBB'=\{b_1,\dots,b_m,c_1,\dots,c_m,d,f,z\}$. Because $\Namings{\pp} = \Namings{\qq}$ by Lemma \ref{LemNamingsClass}, $\Eperm{\AAA} = <\Eperm{\BBB},\Eperm{\BBB'}> \subseteq \Namings{\pp}$ by Lemma \ref{LemCombiningGroups}.
			
			The cases for $n=0$ and $n=1$ may be handled first. We now suppose $n\geq 2$. By performing an $\{a,h_n\}$-switch followed by an $\{h_{n-1},z\}$-switch on $\pp$, we arrive at $\tilde{\qq}$ of the form
				$$ \tilde{\qq} = \mtrx{\LL{h_n}{h_{n-1}}\;\LL{g_n}{g_{n-1}}\;\LL{z}{z}\;\LL{a}{g_n}\;\LL{*}{a}\;\LL{d}{*}\;\LL{e}{f}\;\LL{f}{e}\;\LL{\#}{d}\;\LL{g_{n-1}}{\#}\;\LL{h_{n-1}}{h_n}}$$
			where
				\begin{equation}\label{EqStarsAndPounds}
					\bmtrx{\LL{*}{*}}
						= \bmtrx{\LL{b_1}{c_1}\;\LL{c_1}{b_1}\;\LL{\dots}{\dots}\;\LL{b_m}{c_m}\;\LL{c_m}{b_m}}
					\mbox{ and }
					\bmtrx{\LL{\#}{\#}}
						= \bmtrx{\LL{g_1}{h_1}\;\LL{h_1}{g_1}\;\LL{\dots}{\dots}\;\LL{g_{n-2}}{h_{n-2}}\;\LL{h_{n-2}}{g_{n-2}}}
				\end{equation}
			By direct computation (provided in the appendix), we find $\qq$ in the Rauzy Class of $\tilde{\qq}$ with $\qq = \pp\circ\nu_1$, where
				$$ \nu_1 = \left\{\begin{array}{r}
				    (a,g_{n-1},g_{n-3}\dots g_1,h_1,h_3 \dots h_{n-1},z,g_n,g_{n-2}\dots g_2,e,h_2,h_4\dots h_n) \\
					  \mbox{if }n\mbox{ is even},\\
				    (a,g_{n-1},g_{n-3}\dots g_2,e,h_2,h_4 \dots h_{n-1},z,g_n,g_{n-2}\dots g_1,h_1,h_3\dots h_n) \\
					  \mbox{if }n\mbox{ is odd}.
				           \end{array}\right.$$
			We then see that $\nu_1^{n+1} = (a,z,h_n,h_{n-1},\dots,h_1,e,g_1,g_2,\dots,g_n)$.

			By performing an $\{h_n,z\}$-switch followed by an $\{a,h_{n-1}\}$-switch on $\pp$, we arrive at $\tilde{\rr}$ of the form
				$$ \tilde{\rr} = \mtrx{\LL{h_{n-1}}{h_n}\;\LL{g_n}{g_{n-1}}\;\LL{a}{a}\;\LL{z}{g_n}\;\LL{*}{z}\;\LL{d}{*}\;\LL{e}{f}\;\LL{f}{e}\;\LL{\#}{d}\;\LL{g_{n-1}}{\#}\;\LL{h_n}{h_{n-1}}},$$
				remembering \eqref{EqStarsAndPounds}.
			Again by direct computation, we find $\rr$ in the Rauzy Class of $\tilde{\rr}$ such that $\rr = \pp \circ \nu_2$, where
				$$ \nu_2 = \left\{\begin{array}{r}
				    (a,g_{n},g_{n-2}\dots g_2,e,h_2,h_4 \dots h_{n},z,g_{n-1},g_{n-3}\dots g_1,h_1,h_3\dots h_{n-1}) \\
					  \mbox{if }n\mbox{ is even},\\
				    (a,g_{n},g_{n-2}\dots g_1,h_1,h_3 \dots h_{n},z,g_{n-1},g_{n-3}\dots g_2,e,h_2,h_4\dots h_{n-1}) \\
					  \mbox{if }n\mbox{ is odd}.
				           \end{array}\right.$$
			We see that $\nu_2^{n+1} = (z,a,h_n,h_{n-1}\dots,h_1,e,g_1,g_2\dots g_n)$.

			By Lemma \ref{LemCycleTransposeGenerators},
			    $$ \Eperm{\BBB} = <\nu^{n+1}_1,\nu_2^{n+1}> \subseteq \Namings{\pp},$$
			as $\nu_2^{n+1} = (a,z)\cdot \nu_1^{n+1} \cdot (a,z)$.

			If $m=n=0$, then $\pp$ is actually the form from Lemma \ref{LemOR} with $N=5$, and $\Namings{\pp} = \perm{\AAA}$.
			If $m=n\geq 1$, then we let $\tilde{\uu}$ be the result of a $\{a,b_1\}$-switch on $\pp$. This is of the form
			     $$ \tilde{\uu} = \mtrx{\LL{b_1}{z}\;\LL{c_1}{*'}\;\LL{*'}{f}\;\LL{d}{e}\;\LL{e}{d}\;\LL{f}{\#'}\;\LL{\#'}{a}\;\LL{a}{c_1}\;\LL{z}{b_1}},$$
			     where
				$$ \bmtrx{\LL{*'}{*'}}
						= \bmtrx{\LL{b_2}{c_2}\;\LL{c_2}{b_2}\;\LL{\dots}{\dots}\;\LL{b_m}{c_m}\;\LL{c_m}{b_m}}
					\mbox{ and }
					\bmtrx{\LL{\#'}{\#'}}
						= \bmtrx{\LL{g_1}{h_1}\;\LL{h_1}{g_1}\;\LL{\dots}{\dots}\;\LL{g_{m}}{h_{m}}\;\LL{h_{m}}{g_{m}}}$$
			By direct calculation (provided in appendix), we arrive at $\uu$ in the Rauzy Class of $\tilde{\uu}$ such that $\uu=\pp\circ\mu$ where 
			    $$ \mu = (a,c_1,g_m,c_2,g_{m-1},\dots,c_m,g_1,f,e,d,h_1,b_m,h_2,b_{m-1},\dots,h_m,b_1)$$
			a cycle on $4m+4$ elements and so an odd permutation. $\Namings{\pp}$ contains $\Eperm{\AAA}$ and an odd permutation, and therefore $\Namings{\pp} = \perm{\AAA}$.
		\end{proof}
		
	    \subsection{Multiple Block Permutations}
	    
	    The previous section speaks for very specific pairs. The next lemma allows us to use these results to examine more general pairs, those composed of blocks, by combinations of simpler pairs.

		\begin{lem}\label{LemCombiningBlocks}
		    If $\pp\in\irr{\AAA}$ is of the form
			  $$ \pp = \mtrx{\LL{a}{z}\;B_1\;\LL{s}{s}\;B_2\;\LL{z}{a}}$$
		    where $B_1$ and $B_2$ are blocks, then
			  $$ <\Namings{\qq}, \Namings{\rr}, \Eperm{\{a,s,z\}}>\subseteq \Namings{\pp}$$
		    for
			  $$ \qq = \mtrx{\LL{a}{z}\; B_2 \; \LL{z}{a}}\mbox{ and } \rr = \mtrx{\LL{s}{z}\;B_1\;\LL{z}{s}}.$$
		\end{lem}		
		
		\begin{proof}
			By applying an $\{s,z\}$-switch followed by an $\{a,z\}$-switch, we verify that $(a,s,z)\in \Namings{\pp}$.

			We claim the following: if a switch move $\omega$ acts on $\qq$ with result $\tilde{\qq}$ of the form
						$$ \tilde{\qq} = \mtrx{\LL{\tilde{a}}{\tilde{z}}\; \tilde{B}_2 \; \LL{\tilde{z}}{\tilde{a}}}$$ 
			then by acting on $\pp$ by switch path $\omega'$ that is of length two and begins with $\omega$, we arrive at
						$$ \tilde{\pp} = \mtrx{\LL{\tilde{a}}{\tilde{z}}\; B_1\;\LL{s}{s}\;\tilde{B}_2 \; \LL{\tilde{z}}{\tilde{a}}}.$$
			It then follows by induction on path length that if $\qq\circ \mu \in \ExtClassLab{\qq}$ then $\pp \circ \mu \in \ExtClassLab{\pp}$.
			
			We may verify the claim by cases. If $\omega$ is a $\{b,c\}$-switch ($q_\eps(b)<q_\eps(c)$ for $\eps\in\{0,1\}$) on $\qq$, then $\omega'$ is $\omega$ followed by a $\{c,s\}$-switch. If $\omega$ is a $\{a,b\}$-switch, then $\omega'$ is $\omega$ followed by an $\{a,s\}$-switch. If $\omega$ is a $\{b,z\}$-switch, then $\omega'$ is $\omega$ followed by an $\{s,z\}$-switch.
				
			We may first act on $\pp$ by an $\{a,s\}$-switch to arrive at
				$ \mtrx{\LL{s}{z}\;B_2\;\LL{a}{a}\;B_1\;\LL{z}{s}}.$
			We repeat the above arguments for $\rr$ instead to find
				$ \mtrx{\LL{\tilde{s}}{\tilde{z}}\;B_2\;\LL{a}{a}\;\tilde{B}_1\;\LL{\tilde{z}}{\tilde{s}}}$
			where $\tilde{\rr} = \mtrx{\LL{\tilde{s}}{\tilde{z}}\; \tilde{B}_1 \; \LL{\tilde{z}}{\tilde{s}}}$ for $\tilde{\rr} = \rr\circ \mu$, $\mu \in \Namings{\rr}$. We act on this pair by an $\{a,\tilde{s}\}$-switch to achieve
					$$ \tilde{\pp} = \mtrx{\LL{a}{\tilde{z}}\;\tilde{B}_1\;\LL{\tilde{s}}{\tilde{s}}\;B_2\;\LL{\tilde{z}}{a}} \in\ExtClassLab{\pp}$$
			or $\mu \in \Namings{\pp}$ as desired.
			
		\end{proof}
		
		We now have all necessary components to show the following proposition, finishing the section. In particular, this result shows that $\#\Namings{\pp} \geq N!/2$ for any $\pp\in\irr{\AAA}$ where $\#\AAA = N$.
		
		\begin{prop}\label{PropNamingsContainEven}
		 If $\pp\in\irr{\AAA}$, then $\Eperm{\AAA} \subseteq \Namings{\pp}$.
		\end{prop}
		
		\begin{proof}
			By Lemma \ref{LemNamingsClass} and Proposition \ref{PropForms}, we may assume that
				$$ \pp = \mtrx{\LL{a}{z}\;B_1\;\LL{s_1}{s_1}\;B_2\;\LL{s_2}{s_2}\;\dots\;\LL{s_{m-1}}{s_{m-1}}\;B_m\;\LL{z}{a}}$$
			where each $B_j$ is of one of the Forms in Proposition \ref{PropForms}. The result then follows from applying Lemmas \ref{LemOR}, \ref{LemEvenSpin} and \ref{LemEvenCycles} along with Lemmas \ref{LemCombiningBlocks} and \ref{LemCombiningGroups} by induction on $m$.
		\end{proof}

	 \section{Containing the Symmetric Group}

		  \begin{prop}\label{PropFullGroup}
			Let $\pp\in\irr{\AAA}$. Then $\Namings{\pp} \subseteq \Eperm{\AAA}$ if and only if $\Pof{\pp}$ is simple.
		  \end{prop}
		  
		  \begin{proof}
		  		We may show this by cases, based on the Type in Proposition \ref{PropTypes}. We begin each case with a convenient choice of $\qq$ in the Extended Rauzy Class of $\pp$, remembering that $\Namings{\qq} = \Namings{\pp}$ by Lemma \ref{LemNamingsClass} and $\Pof{\qq}=\Pof{\pp}$ by Corollary \ref{CorPofInvariant}. We will show that $\Namings{\pp}$ contains an odd permutation if and only if $\Pof{\pp}$ has a repeated value.
		  		
		  		\textit{Type \ref{TypeHyp} (Hyperelliptic)}: Consider $\qq$ of the form
		  			$$ \qq = \mtrx{\LL{a}{z}\; \LL{b_1}{b_n}\;\LL{b_2}{b_{n-1}}\;\LL{\dots}{\dots}\;\LL{b_n}{b_1}\;\LL{s_1\dots s_k}{s_1\dots s_k} \;\LL{z}{a}},$$
		  			$n\geq 2$ and $k\geq 0$, in the Extended Rauzy Class of $\pp$. If $\Pof{\pp}$ has a repeated value, then either $k>1$ or $n$ is odd. By Lemma \ref{LemOR} or Lemma \ref{LemEmpty} along with Lemma \ref{LemCombiningBlocks} and \ref{LemCombiningGroups}, $\Namings{\pp} = \perm{\AAA}$.
		  			
		  			If $\Pof{\pp}$ has no repeated values, then $k\leq 1$ and $n$ is even. Then $\Zof{\qq}\subseteq\Eperm{\AAA}$, as it is generated by the cycles of odd order. By Corollary \ref{CorEvenZof}, $\Namings{\pp}\subseteq\Eperm{\AAA}$.
		  			
		  \textit{Type \ref{TypeEvenOdd} (Odd Cycles, Odd Spin)}: If $\Pof{\pp}$ has no repeated values, then $\Zof{\pp}\subseteq\Eperm{\AAA}$. Therefore $\Namings{\pp}\subset\Eperm{\AAA}$ by Corollary \ref{CorEvenZof}.
		  
		  	If $\Pof{\pp}$ has a repeated value $2\ell +1$, then there exists $\qq$ in the Extended Rauzy Class of $\pp$ of the form
		  		$$ \qq = \mtrx{\LL{a}{z}\;B\;\LL{b_1}{c_1}\;\LL{c_1}{b_1}\;\LL{\dots}{\dots}\;\LL{b_\ell}{c_\ell}\;\LL{c_\ell}{b_\ell}\;\LL{t}{t}\;\LL{d_1}{e_1}\;\LL{e_1}{d_1}\;\LL{\dots}{\dots}\;\LL{d_\ell}{e_\ell}\;\LL{e_\ell}{d_\ell}\;\LL{z}{a}},$$
		  		where $B$ is either an empty block or has the same last letter on each row, which we will call $s$. Then by performing a $\{t,z\}$-switch (and then an $\{s,z\}$-switch if $s$ exists, i.e. $B\neq \emptyset$), we arrive at
		  			$$ \mtrx{\LL{a}{t}\;B\;\LL{d_1}{e_1}\;\LL{e_1}{d_1}\;\LL{\dots}{\dots}\;\LL{d_\ell}{e_\ell}\;\LL{e_\ell}{d_\ell}\;\LL{z}{z}\;\LL{b_1}{c_1}\;\LL{c_1}{b_1}\;\LL{\dots}{\dots}\;\LL{b_\ell}{c_\ell}\;\LL{c_\ell}{b_\ell}\;\LL{t}{a}}.$$
		  	So $\nu = (b_1,d_1)(b_2,d_2)\dots(b_\ell,d_\ell)(c_1,e_1)(c_2,e_2)\dots(c_\ell,e_\ell)(t,z)$, an odd permutation, belongs to $\Namings{\pp}$. Therefore $\Namings{\pp}=\perm{\AAA}$.
		  	
		  	\textit{Type \ref{TypeTwosEven} (3/1 Cycles, Even Spin)}: Note that $\Pof{\pp} = \{3^n,1^m\}$ in this case ($n\geq 2$ and $m\geq 0$), so $3$ is always a repeated value. Because there exists $\qq$ in the Rauzy Class of $\pp$ that contains a length 5 block of Form \ref{FormOR}, $\Namings{\pp}=\perm{\AAA}$ by Lemma \ref{LemCombiningBlocks} and \ref{LemOR}.
		  	
		  	\textit{Type \ref{TypeEvenEven} (Odd Cycles, Even Spin)}: This case follows an almost identical argument to Type \ref{TypeEvenOdd}. We remark that if a block of Form \ref{FormFourAndTwos} (of length $L_1\geq 4$) and one of Form \ref{FormTwos} exists (of length $L_2\geq 4$) in $\qq\in\irr{\AAA}$ of this Type, then there exists $\tilde{\qq}$ in the Rauzy Class of $\qq$ that is identical except that the block of length $L_1$ is now Form \ref{FormTwos} and the block of length $L_2$ is now Form \ref{FormFourAndTwos}.
		  	
		  	\textit{Type \ref{TypeOdd} (Even Cycles)}: Suppose $\Pof{\pp}$ contains a repeated value $\ell$. If $\ell$ is odd, then we may find $\qq$ in the Extended Rauzy Class with two Form \ref{FormTwos} blocks of length $\ell-1$ as mentioned in the Type \ref{TypeEvenOdd} case, and so $\Namings{\pp}=\perm{\AAA}$. If $\ell$ is odd, we may find $\qq$ in the Extended Rauzy Class of $\pp$ that contains a block of Form \ref{FormThreeAndTwos} with $n=m=(\ell-1)/2$. By Lemma \ref{LemEvenCycles}, $\Namings{\pp}=\perm{\AAA}$.

			Now suppose $\Pof{\pp}$ is simple. Then we may find by Proposition \ref{PropSufficient} $\qq$ of the form in Lemma \ref{LemCombinedInserts}, or
				$$ \qq = \mtrx{\LL{a}{t_n}\;B_1\; \LL{s_1}{s_1}\; B_2\;\LL{s_2}{s_2}\;\dots\; B_m\;\LL{s_m}{s_m}\;C_1\;\LL{t_1}{t_1}\;\dots\; \LL{t_{n-1}}{t_{n-1}}\; C_n \;\LL{t_n}{a}},$$
			with sub-alphabets $\BBB_1$, $\dots$, $\BBB_m$, $\CCC_1^\pm$, $\dots$, $\CCC_n^{\pm}$ as defined in that lemma. We may further assume that
				\begin{equation}
					\#\BBB_j > \#\BBB_{j+1}, ~ \#\CCC_j^-> \#\CCC_j^+\mbox{ and }\#\CCC_j^+ > \#\CCC_{j+1}^-,
				\end{equation}
			as $\Pof{\qq}$ is simple. Note that these sub-alphabets are uniquely determined by the cycles of $\Nof{\qq}$.
			
			We fix $\XXX = \{x_1,\dots,x_{m+2n-1}\}$ of $m+2n-1$ symbols so that $\AAA\cap\XXX=\emptyset$, and we also fix choices $\overline{b} = (b_1,\dots,b_{m+2n-1})$ and $\overline{c} = (c_1,\dots,c_{m+2n-1})$ that satisfy the conditions in Lemma \ref{LemCombinedInserts}. Let $\hat{\qq} = \Ext{\overline{x}}{\overline{b}}{\overline{c}}{\qq}$, where $\overline{x} = (x_1,\dots,x_{m+2n-1})$.
			
			We will now show that $\Namings{\qq}\subseteq\Eperm{\AAA}$. Fix any $\nu \in \Namings{\qq}$ an let $\rr = \qq \circ \nu$ be the resulting pair in $\ExtClassLab{\qq}$. We define $\overline{b'} = (b_1',\dots,b_{m+2n-1}')$ and $\overline{c'}=(c_1',\dots,c_{m+2n-1}')$ to satisfy
				\begin{equation}
					b_j = \nu b_j'\mbox{ and } c_j = \nu c_j'.
				\end{equation}
			It follows then that $\overline{b'}$ and $\overline{c'}$ satisfy the conditions in Lemma \ref{LemCombinedInserts} for $\rr$. Let $\hat{\rr}:= \Ext{\overline{x}}{\overline{b'}}{\overline{c'}}{\rr}$. We note by choice of insertion rules that $\hat{\rr} = \hat{\qq}\circ \nu$ and so
				\begin{equation}\label{EqRandQ}
					\Nof{\hat{\rr}} = \Nof{\hat{\qq}}* \nu\mbox{ and } \Spinof{\hat{\rr}} = \Spinof{\hat{\qq}}.
				\end{equation}
			
			Because $\rr \in \ExtClassLab{\qq}$, we may choose a switch path $\omega$ from $\qq$ to $\rr$. We consider the lifted switch path $\hat{\omega}$ from $\hat{\qq}$ and call the resulting pair $\hat{\uu}$. We recall Proposition \ref{PropSpinDefined} and $\mu_{\hat{\omega}}\in\Eperm{\AAA\cup\XXX}$ from Definition \ref{DefMuOmega} to see that
				\begin{equation}\label{EqQandU}
					\Nof{\hat{\uu}} = \Nof{\hat{\qq}} * \mu_{\hat{\omega}}\mbox{ and }\Spinof{\hat{\uu}} = \Spinof{\hat{\qq}}.
				\end{equation}
				
			By following the computations in Lemma \ref{LemSwitchAndInsert}, we see that there exist a choice of $\overline{b''}= (b_1'',\dots,b_{m+2n-1}'')$ and $\overline{c''} = (c_1'',\dots,c_{m+2n-1}'')$ so that $\hat{\uu} = \Ext{\overline{x}}{\overline{b''}}{\overline{c''}}{\rr}$. Furthermore, we may verify that $\overline{b''}$ an $\overline{c''}$ satisfy the conditions of Lemma \ref{LemCombinedInserts}. Because $\Spinof{\hat{\uu}} = \Spinof{\hat{\rr}}$,
				\begin{equation}\label{EqRandU}
					\Nof{\hat{\rr}} = \Nof{\hat{\uu}} * \eta,~\eta \in \Eperm{\AAA\cup\XXX}
				\end{equation}
			by the final result in Lemma \ref{LemCombinedInserts}. We combine Equations \eqref{EqRandQ}--\eqref{EqRandU} to see that
				\begin{equation}
					\nu = \mu_{\hat{\omega}}\cdot \eta\cdot \rho,~\rho\in\Zof{\hat{\qq}}.
				\end{equation}
			Because $\Zof{\hat{\qq}}\subseteq\Eperm{\AAA\cup\XXX}$, $\nu$ is even.
		  \end{proof}

	\appendix
	
	\section{Algebra: Generators of the Symmetric and Alternating Groups}
	
			In this appendix, we provide lemmas related to generating elements of $\perm{\AAA}$ and $\Eperm{\AAA}$. We leave all proofs as exercises.

		  	\begin{lem}\label{LemEvenOddGenerators}
				If $\AAA=\{1,2,\dots,N\}$ and
					$$\nu_1 = (1,2,3\dots,N-1),\nu_2 = (2,3,\dots,N)\in \perm{\AAA},$$
				then
			      $$ <\nu_1,\nu_2> = \RHScase{\Eperm{\AAA}, & \mbox{if }N\mbox{ is even}, \\ \perm{\AAA}, & \mbox{if }N\mbox{ is odd}.}$$
			\end{lem}

		\begin{lem}\label{LemCycleTransposeGenerators}
			If $\AAA = \{1,\dots,N\}$, $\nu_1 = (1,2,\dots,N)$ and $\nu_2 = (1,2)$, then
				$$ <\nu_1,\nu_2> = \perm{\AAA}.$$
			If $N$ is odd, then
				$$ <\nu_1, \nu_2\nu_1\nu_2> = \Eperm{\AAA}.$$
		\end{lem}
	    
		\begin{lem}\label{LemCombiningGroups}
		      Suppose $\BBB,\CCC$ are each alphabets on at least $3$ letters and $\BBB\cap\CCC\neq \emptyset$. If $\AAA = \BBB\cup \CCC$, then
			  $$ <\Eperm{\BBB},\Eperm{\CCC}> = \Eperm{\AAA} \mbox{ and } <\Eperm{\BBB},\perm{\CCC}> = \perm{\AAA}.$$
		\end{lem}

	\section{Switch Moves}
	
		In this appendix, we provide the combinatorial work mentioned in the main sections. These utility results follow in format to those in \cite{cFick2012}. As in that work, we will speak of \emph{projection} onto sub-alphabets $\AAA'\subsetneq \AAA$. In particular, suppose
				\begin{equation}
					\pp = \mtrx{\LL{a}{z}\;B_1\;\LL{s}{s}\;B_2\;\LL{z}{a}}
				\end{equation}
		and let $\BBB_j$ be the letters in block $B_j$, $j\in\{1,2\}$. By considering $\pp$ projected on $\AAA\setminus\big(\BBB_j\cup \{s\}\big)$, $j\in\{1,2\}$, we arrive at $\pp^{(j)}$ given by
				\begin{equation}
					\pp^{(1)} = \mtrx{\LL{a}{z}\;B_2\;\LL{z}{a}} \mbox{ and }\pp^{(2)} = \mtrx{\LL{a}{z}\;B_1\;\LL{z}{a}}.
				\end{equation}
		We may act on either $\pp^{(j)}$ by (inner) switch moves to arrive at $q^{(j)}\in\ClassLab{\pp^{(j)}}$ of the form
				\begin{equation}
					\qq^{(1)} = \mtrx{\LL{a}{z}\;B_2'\;\LL{z}{a}} \mbox{ and }\qq^{(2)} = \mtrx{\LL{a}{z}\;B_1'\;\LL{z}{a}}
				\end{equation}
		for blocks $B_j'$. As shown in \cite{cFick2012}, the corresponding lifts
				\begin{equation}
					\mtrx{\LL{a}{z}\;B_1\;\LL{s}{s}\;B_2'\;\LL{z}{a}}\mbox{ and }\mtrx{\LL{a}{z}\;B_1'\;\LL{s}{s}\;B_2\;\LL{z}{a}}
				\end{equation}
		belong to $\ClassLab{\pp}$, as they may be achieved by following the same switch moves as in the projections (followed a correcting switch). This fact will be used freely in these proofs.

	    \begin{lem}\label{LemMovesTwosInTwo2}
		  If $\tilde{\pp}\in\irr{\AAA}$ is of the form
		      $$ \tilde{\pp} = \mtrx{\LL{a}{z}\;\LL{b}{c}\;\LL{*}{b}\;\LL{c}{*}\;\LL{z}{a}}$$
		  where
		      $$ \left[\LL{*}{*}\right] = \left[\LL{d_1}{e_1}\;\LL{e_1}{d_1}\;\LL{\dots}{\dots}\;\LL{d_n}{e_n}\;\LL{e_n}{d_n}\right],$$
		  for $n\geq 1$, then $\pp$ exists in the Rauzy Class of $\tilde{\pp}$ of the form
		      $$ \pp = \mtrx{\LL{a}{z}\;\LL{e_n}{c}\;\LL{c}{e_n}\;\LL{e_{n-1}}{d_n}\;\LL{d_n}{e_{n-1}}\;\LL{\dots}{\dots}\;\LL{e_1}{d_2}\;\LL{d_2}{e_1}\;\LL{b}{d_1}\;\LL{d_1}{b}\;\LL{z}{a}}.$$
	    \end{lem}
	    
	    \begin{proof}
	    	This may be proved by induction on $n$, and we will provide the inductive step. We act on
	    		$$ \tilde{\pp} = \mtrx{\LL{a}{z}\;\LL{b}{c}\;\LL{*'}{b}\;\LL{d_n}{*'}\;\LL{e_n}{e_n}\;\LL{c}{d_n}\;\LL{z}{a}}$$
	    	by a $\{b,d_n\}$-switch followed by a $\{c,e_n\}$-switch to arrive at
	    		$$ \mtrx{\LL{a}{z}\;\LL{c}{d_n}\;\LL{*'}{c}\;\LL{d_n}{*'}\;\LL{b}{e_n}\;\LL{e_n}{b}\;\LL{z}{a}}$$
	    	which, after projecting onto $\AAA\setminus\{b,e_n\}$, is of the form in the lemma with $n-1$.
	    \end{proof}

	    \begin{lem}\label{LemMovesTwosInTwo}
		  If $\tilde{\pp}\in\irr{\AAA}$ is of the form
		      $$ \tilde{\pp} = \mtrx{\LL{a}{z}\;\LL{b}{*}\;\LL{*}{c}\;\LL{c}{b}\;\LL{z}{a}}$$
		  where
		      $$ \left[\LL{*}{*}\right] = \left[\LL{d_1}{e_1}\;\LL{e_1}{d_1}\;\LL{\dots}{\dots}\;\LL{d_n}{e_n}\;\LL{e_n}{d_n}\right],$$
		  for $n\geq 1$, then $\pp$ exists in the Rauzy Class of $\tilde{\pp}$ of the form
		      $$ \pp = \mtrx{\LL{a}{z}\;\LL{e_n}{c}\;\LL{c}{e_n}\;\LL{e_{n-1}}{d_n}\;\LL{d_n}{e_{n-1}}\;\LL{\dots}{\dots}\;\LL{e_1}{d_2}\;\LL{d_2}{e_1}\;\LL{b}{d_1}\;\LL{d_1}{b}\;\LL{z}{a}}.$$
	    \end{lem}
	    
	    \begin{proof}
	    	The proof is very similar to that of Lemma \ref{LemMovesTwosInTwo2}.
	    \end{proof}

	    \begin{lem}\label{LemMovesToThreeAndTwos}
		  If $\tilde{\pp}\in\irr{\AAA}$ is of the form
		      $$ \tilde{\pp} = \mtrx{\LL{a}{z}\;\LL{b}{d}\;\LL{c}{*}\;\LL{d}{c}\;\LL{e}{b}\;\LL{*}{f}\;\LL{f}{e}\;\LL{z}{a}}$$
		  where
		      $$ \left[\LL{*}{*}\right] = \left[\LL{g_1}{h_1}\;\LL{h_1}{g_1}\;\LL{\dots}{\dots}\;\LL{g_n}{h_n}\;\LL{h_n}{g_n}\right],$$
		  for $n\geq 1$, then $\pp$ exists in the Rauzy Class of $\tilde{\pp}$ of the form
		      $$ \pp = \mtrx{\LL{a}{z}\;\LL{b}{d}\;\LL{c}{c}\;\LL{d}{b}\;\LL{h_n}{f}\;\LL{f}{h_n}\;\LL{h_{n-1}}{g_n}\;\LL{g_n}{h_{n-1}}\;\LL{\dots}{\dots}\;
			    \LL{h_1}{g_2}\;\LL{g_2}{h_1}\;\LL{e}{g_1}\;\LL{g_1}{e}\;\LL{z}{a}}.$$
	    \end{lem}

	    \begin{proof}
		We perform two switch moves on
		      $$ \tilde{\pp} = \mtrx{\LL{a}{z}\;\LL{b}{d}\;\LL{c}{h_1}\;\LL{d}{g_1}\;\LL{e}{*'}\;\LL{g_1}{c}\;\LL{h_1}{b}\;\LL{*'}{f}\;\LL{f}{e}\;\LL{z}{a}}$$
		and leave the rest of the proof to follow by induction. We perform an $\{f,g_1\}$-switch followed by a $\{d,f\}$-switch to arrive at
		      $$ \mtrx{\LL{a}{z}\;\LL{b}{d}\;\LL{c}{*'}\;\LL{d}{c}\;\LL{h_1}{b}\;\LL{*'}{f}\;\LL{f}{h_1}\;\LL{e}{g_1}\;\LL{g_1}{e}\;\LL{z}{a}}$$
		which, by projecting to $\AAA\setminus\{e,g_1\}$, is a permutation of the same form as in the lemma with $n-1$ rather than $n$.
	    \end{proof}

	    \begin{lem}\label{LemMovestoTwosAndThree}
		  If $\tilde{\pp}\in\irr{\AAA}$ is of the form
		      $$ \tilde{\pp} = \mtrx{\LL{a}{z}\;\LL{b}{*}\;\LL{c}{d}\;\LL{*}{c}\;\LL{d}{b}\;\LL{z}{a}}$$
		  where
		      $$ \left[\LL{*}{*}\right] = \left[\LL{e_1}{f_1}\;\LL{f_1}{e_1}\;\LL{\dots}{\dots}\;\LL{e_n}{f_n}\;\LL{f_n}{e_n}\right],$$
		  for $n\geq 1$, then $\pp$ exists in the Rauzy Class of $\tilde{\pp}$ of the form
		      $$ \pp = \mtrx{\LL{a}{z}\;\LL{f_n}{d}\;\LL{d}{f_n}\;\LL{f_{n-1}}{e_n}\LL{e_n}{f_{n-1}}\;\LL{\dots}{\dots}\;\LL{f_1}{e_2}\;\LL{e_2}{f_1}\;\LL{b}{e_1}\;\LL{c}{c}\;\LL{e_1}{b}\;\LL{z}{a}}.$$
	    \end{lem}

	    \begin{proof}
		We first perform a $\{d,e_1\}$-switch on
		      $$ \tilde{\pp} = \mtrx{\LL{a}{z}\;\LL{b}{f_1}\;\LL{c}{e_1}\;\LL{e_1}{*'}\;\LL{f_1}{d}\;\LL{*'}{c}\;\LL{d}{b}\;\LL{z}{a}}$$
		and arrive at
		      $$ \mtrx{\LL{a}{z}\;\LL{f_1}{*'}\;\LL{*'}{d}\;\LL{d}{f_1}\;\LL{b}{e_1}\;\LL{c}{c}\;\LL{e_1}{b}\;\LL{z}{a}}.$$
		If $n=1$, this is $\pp$. If $n>1$, we apply Lemma \ref{LemMovesTwosInTwo} by projecting to $\AAA\setminus\{b,c,e_1\}$.
	     
	    \end{proof}
	    
	\bibliographystyle{abbrv}
	\bibliography{../../bibfile2}
		
\end{document}